\documentclass[a4paper, 10pt]{article}

\usepackage[a4paper, total={6in, 9.5in}]{geometry}

\usepackage[utf8]{inputenc}
\usepackage{amssymb}
\usepackage{amsthm}
\usepackage{amsmath}
\usepackage{amsfonts}

\usepackage{graphicx}

\usepackage{comment}
\usepackage{titling}
\usepackage{tikz-cd}

\usepackage{graphicx}
\usepackage{hyperref}
\usepackage{mathrsfs}
\usepackage{mathtools}

\usepackage{stmaryrd}

\usepackage{float}

\usepackage{enumitem}

\newcounter{cislo} \numberwithin{cislo}{section}

\numberwithin{equation}{section}

\newtheorem*{theorem*}{Theorem}
\newtheorem{theorem}[cislo]{Theorem}
\newtheorem{lemma}[cislo]{Lemma}
\newtheorem{proposition}[cislo]{Proposition}
\newtheorem{corollary}[cislo]{Corollary}

\theoremstyle{definition}
\newtheorem{definition}[cislo]{Definition}
\newtheorem{notation}[cislo]{Notation}
\newtheorem{setup}[cislo]{Setup}

\newtheorem{discussion}[cislo]{Discussion}
\newtheorem{properties}[cislo]{Properties}

\newtheorem{construction}[cislo]{Construction}

\newtheorem{remark}[cislo]{Remark}
\newtheorem{example}[cislo]{Example}

\theoremstyle{remark}

\DeclareMathOperator{\rep}{rep}
\DeclareMathOperator{\G}{G}

\DeclareMathOperator{\rec}{rec}
\DeclareMathOperator{\ab}{ab}

\DeclareMathOperator{\cd}{cd}
\DeclareMathOperator{\lvl}{lvl}
\DeclareMathOperator{\sg}{sg}
\DeclareMathOperator{\subcat}{sub}
\DeclareMathOperator{\LL}{LL}

\newcommand{\fd}{\mathrm{fd}}

\DeclareMathOperator{\W}{\mathcal{W}}

\DeclareMathOperator{\K}{K}
\DeclareMathOperator{\D}{D}
\DeclareMathOperator{\eL}{L}

\DeclareMathOperator{\RG}{R\Gamma}
\DeclareMathOperator{\KG}{K\Gamma}
\DeclareMathOperator{\GG}{G\Gamma}

\DeclareMathOperator{\FinSet}{FinSet}

\DeclareMathOperator{\sh}{sh}
\DeclareMathOperator{\add}{add}

\newcommand{\A}{A}
\newcommand{\B}{B}

\DeclareMathOperator{\Sh}{Sh}

\DeclareMathOperator{\Hom}{Hom}
\DeclareMathOperator{\Aut}{Aut}

\DeclareMathOperator{\Rep}{Rep}

\DeclareMathOperator{\Ind}{Ind}
\DeclareMathOperator{\cInd}{cInd}
\DeclareMathOperator{\Res}{Res}

\DeclareMathOperator{\Irr}{Irr}

\DeclareMathOperator{\Grp}{Grp}

\newcommand{\integral}{\mathrm{int}}
\newcommand{\gen}{\mathrm{gen}}

\DeclareMathOperator{\Mo}{Mod}
\DeclareMathOperator{\mo}{mod}

\let\lim\relax
\DeclareMathOperator*{\lim}{lim}

\DeclareMathOperator{\id}{id}

\DeclareMathOperator{\ft}{ft}

\DeclareMathOperator{\Gal}{Gal}

\DeclareMathOperator{\eF}{\mathscr{F}}

\DeclareMathOperator{\F}{\mathbb{F}}
\DeclareMathOperator{\Z}{\mathbb{Z}}
\DeclareMathOperator{\Q}{\mathbb{Q}}
\DeclareMathOperator{\N}{\mathbb{N}}

\let\O\relax
\DeclareMathOperator{\O}{\mathcal{O}}

\DeclareMathOperator{\pt}{pt}

\DeclareMathOperator{\Ql}{\overline{\mathbb{Q}}_\ell}
\DeclareMathOperator{\Zl}{\overline{\mathbb{Z}}_\ell}
\DeclareMathOperator{\Fl}{\overline{\mathbb{F}}_\ell}

\usepackage[style=alphabetic]{biblatex}
\renewbibmacro{in:}{}

\begin{document}

\title{On modulo $\ell$ cohomology  \\ of $p$-adic Deligne--Lusztig varieties for $GL_n$}
\author{Jakub Löwit}
\date{}

\maketitle
\begin{abstract}
In 1976, Deligne and Lusztig realized the representation theory of finite groups of Lie type inside étale cohomology of certain algebraic varieties. Recently, a $p$-adic version of this theory started to emerge: there are $p$-adic Deligne--Lusztig spaces, whose cohomology encodes representation theoretic information for $p$-adic groups -- for instance, it partially realizes the local Langlands correspondence with characteristic zero coefficients. 
However, the parallel case of coefficients of positive characteristic $\ell \neq p$ has not been inspected so far. 
The purpose of this article is to initiate such an inspection. In particular, we relate cohomology of certain $p$-adic Deligne--Lusztig spaces to Vignéras's modular local Langlands correspondence for $\mathbf{GL}_n$.
\end{abstract}

\section{Introduction}
\label{section: introduction}

\subsection{Context and history}

The realization of representations of finite groups of Lie type -- groups of points of reductive groups $\mathbf{G}$ over $\F_q$ -- in the cohomology of the so-called Deligne--Lusztig varieties dates back to \cite{DL76}. At that time, the desire was to describe representations of the finite group $G = \mathbf{G}(\F_{q})$ over an algebraically closed field of characteristic zero\footnote{Such coefficient fields are indistinguishable from the point of view of first order logic, so the representation theory is independent of the specific choice.}.
After certain choices, one explicitly produces a variety $\dot{X}$ with an action of the finite group $G \times T  = \mathbf{G}(\F_q) \times \mathbf{T}(\F_q)$ where $\mathbf{T}$ is an $\F_q$-rational torus of $\mathbf{G}$. The $\ell$-adic cohomology $H^\bullet_c(\dot{X}, \overline{\Q}_\ell)$ for $\ell \neq p$ inherits this action -- one can regard its Euler characteristic $[H^\bullet_c(\dot{X}, \overline{\Q}_\ell)]$ as an element of the Grothendieck group $\G_0(G\times T, \overline{\Q}_\ell)$ of finite-dimensional representations. This decomposes into isotypic components labelled by characters of $T$, and the procedure of picking such a weight space relates characters of $T$ with representations of $G$. The categories $\rep(G, \overline{\Q}_\ell)$ and $\rep(T, \overline{\Q}_\ell)$ of finite-dimensional representations are semisimple, and the results of classical Deligne--Lusztig theory make the above relation meaningful and computable.

There are two interesting directions in which one can try to generalize this theory. Firstly, we can consider reductive groups $\mathbf{G}$ over a non-archimedean local field $K$ instead of $\F_q$, and seek a description of smooth representations of the locally profinite group $\mathbf{G}(K)$ with coefficients in an algebraically closed field of characteristic zero. The existence of a parallel theory was conjectured already by Lusztig, and its development is still in progress. The consequent relation between smooth characters of $\mathbf{T}(K)$ and smooth representations of $\mathbf{G}(K)$ should be tightly connected to the local Langlands correspondence. See \cite{Lus79, Lus04, Boy12, BW16, Iva18, Iva19, Iva20a, Iva20b}.

Secondly, one can consider different coefficient rings $\Lambda$, for instance an algebraically closed field of positive characteristic $\ell\neq p$. This becomes interesting already in the classical case -- when $\ell \mid |G|$, the category $\rep(G, \overline{\F}_\ell)$ ceases to be semisimple. Suddenly, the Grothendieck group does not contain enough information to fully describe the representation theory of $G$. A natural way to circumvent these issues is to work in the finer setting of the bounded derived category $\D^b(G, \Lambda)$, as in \cite{BR02, BDR16, Dud17}. Nevertheless, the semisimplified information is interesting.

In this paper, we are interested in the combination of the two directions above: we consider the coefficients $\Lambda=\overline{\F}_\ell$ of positive characteristic $\ell \neq p$ for the $p$-adic Deligne--Lusztig theory. In doing so, we necessarily end up phrasing certain arguments on the level of derived categories. 

In particular, we carry forward the results of \cite{Iva19} concerning the partial realization of the local Langlands correspondence for $\mathbf{GL}_n$ to the modular setting with coefficients $\Lambda = \overline{\F}_\ell$ on the level of Grothendieck groups. This has two main ingredients -- the results of \cite{Iva19} and the comparison to modular local Langlands correspondence constructed by Vignéras \cite{Vig96, Vig00a, Vig00b}.

\subsection{Our approach and results}

Let $K$ be a non-archimedean local field with residue field $\F_q$ of characteristic $p$. In other words, $K$ is either a finite extension of $\Q_p$ with residue field $\F_q$, or the function field $\F_q(\!(\varpi)\!)$.
Let $\ell \neq p$ be another prime. Consider the group $\mathbf{G} = \mathbf{GL}_n$ over $K$ with $G = \mathbf{G}(K)$ and its maximal torus $\mathbf{T}$, whose $K$-points $T = \mathbf{T}(K)$ correspond to the multiplicative group $L^\times$ of the degree $n$ unramified extension $L/K$. 

To sketch our approach and results, consider the following diagram.
\begin{figure}[H]
\centering
\begin{tikzcd}
\pm [H^\bullet_c(\dot{X}, \overline{\Q}_\ell)_{\theta}] & \mathscr{A}^0_n(K, \overline{\Q}_\ell) \arrow[dotted]{r}{r_\ell} & \mathscr{A}^0_n(K, \overline{\F}_\ell) & \pm [H^\bullet_c(\dot{X}, \overline{\F}_\ell)_{\eL \psi}] \\
\theta \arrow[mapsto]{u} \arrow[mapsto]{d} & \mathscr{X}(L, \overline{\Q}_\ell)^{\sg} \arrow{d} \arrow{u} \arrow[dotted]{r}{r_\ell}  & \mathscr{X}(L, \overline{\F}_\ell)^{\sg} \arrow{d} \arrow{u} & \psi \arrow[mapsto]{u} \arrow[mapsto]{d} \\
\cInd_{\W_L}^{\W_K}(\mu \cdot \theta) & \mathscr{G}^0_n(K, \overline{\Q}_\ell) \arrow[dotted]{r}{r_\ell} & \mathscr{G}^0_n(K, \overline{\F}_\ell) & \cInd_{\W_L}^{\W_K} (\mu \cdot \psi)
\end{tikzcd}
\caption{The reduction diagram.}\label{diagram}
\end{figure}
The left-hand side of Diagram \ref{diagram} takes values in the characteristic zero field $\overline{\Q}_\ell$, while the parallel right-hand side lives over the field $\overline{\F}_\ell$ of positive characteristic. The middle row parametrizes smooth characters of the multiplicative group $L^\times$ in strongly general position\footnote{The notion of {\it strongly general position} is a strengthening of the notion of {\it general position}; see Definition \ref{stronglygeneralposition}.} in each of the characteristics. From here, we can go to the bottom row, which parametrizes $n$-dimensional smooth irreducible representations of the Weil group $\W_K$. This goes by adjusting the given character by the so-called rectifier $\mu$, inflating and inducing -- it is explicit and well understood.

The passage to the irreducible supercuspidal representations of $G$ parametrized by the upper row is the interesting part. Here we consider a suitable Deligne--Lusztig space $\dot{X}$ equipped with commuting actions of $G$ and $T$, take its cohomology, look at the isotypic component\footnote{In characteristic zero, this is given by picking the weight space of $\theta$. In positive characteristic, one needs to do this on the derived level -- the naive definition of isotypic part yields and extra multiplicity for small $\ell$.}
of the character of $T$ we started with, and consider its Euler characteristic in the corresponding Grothendieck group. For characteristic zero coefficients, this is done in \cite{Iva18, Iva19, Iva20a}, leading to the following.

\begin{theorem}[{\cite[Theorem A]{Iva19}}]\label{thmiva}
Assume $p > n$. Let $\theta \in \mathscr{X}(L, \overline{\Q}_\ell)^{\sg}$ be a character in strongly general position. Then $\pm [H^{\bullet}_c(\dot{X}, \overline{\Q}_\ell)_\theta]$ is up to sign an irreducible supercuspidal representation of $G$ and the left-hand side of Diagram \ref{diagram} partially realizes the local Langlands correspondence.
\end{theorem}

Now consider the horizontal {\it reduction mod $\ell$} maps $r_\ell$ in Diagram \ref{diagram}. These make sense after additional technical assumptions and are well-defined only on the level of Grothendieck groups. At this cost, we can prove clean compatibility statements about the parallel sides of the diagram. 
The results of Vignéras \cite{Vig96, Vig00a, Vig00b} further show that for $\mathbf{GL}_n$, the local Langlands correspondence can be partially reduced via $r_\ell$ to the modular setting. Putting these compatibilities together with the above theorem, we obtain our main result.

\begin{theorem}[{Theorem \ref{theorem: upshot}}]\label{thmintro}
Assume $p > n$. Let $\psi \in \mathscr{X}(L, \overline{\F}_\ell)^{\sg}$ be a character in strongly general position. Then $\pm [H^{\bullet}_c(\dot{X}, \overline{\F}_\ell)_{\eL \psi}]$ is up to sign an irreducible supercuspidal representation of $G$, and the right-hand side of Diagram \ref{diagram} partially realizes the modular local Langlands correspondence.
\end{theorem}

Taking the isotypic part of $\psi$ naively results in an extra multiplicity for those $\ell$ which are small with respect to $G$. We make this explicit in Theorem \ref{theorem: naive upshot}.

\subsection{Consequences and further questions}

Our Theorem \ref{thmintro} shows that the cohomology of $p$-adic Deligne--Lusztig spaces for $\mathbf{GL}_n$ encodes a big part of modular representation theory of the locally profinite group $G = \mathbf{GL}_n(K)$.  

The employed compatibilities between the sides of Diagram \ref{diagram} are quite general, so one could hope to apply them for other reductive groups $\mathbf{G}$ where the theory is less developed. In such a setting, there is a good notion of $p$-adic Deligne--Lusztig spaces \cite{Iva20b}. 
However, the modular local Langlands correspondence -- let alone its relation to the characteristic zero case -- is not understood. A better understanding of $H^{\bullet}_c(\dot{X}, \overline{\F}_\ell)$ may be fruitful.

This leads to another question: can we understand the cohomology $H^{\bullet}_c(\dot{X}, \overline{\F}_\ell)$ already on the level of the bounded derived category $\D^b(G, \overline{\F}_\ell)$ of $G$-representations? Indeed, such an understanding is necessary for mimicking the geometric arguments of characteristic zero Deligne--Lusztig theory in the modular setting. 
Already on the level of classical Deligne--Lusztig theory with modular coefficients, this carries important information \cite{BR02, BDR16, Dud17}.

The geometric methods are close in spirit to the recent construction \cite{FS21}, while having more explicit features. One could hope to use them to get a better handle on this approach.

\subsection{Structure of this article}
After the current \S \ref{section: introduction}, this article is structured as follows. 
In \S \ref{section: grothendieck groups of smooth representations} we discuss the smooth representation theory of locally profinite groups. We recall the notion of reduction modulo $\ell$, relating characteristic zero and characteristic $\ell$ coefficients.
In \S \ref{section: the structure of T and its characters} we review the smooth character theory of the $T = L^\times$ and discuss the notion of strongly general position of characters.
In \S \ref{section: local langlands correspondence} we recall the modular local Langlands correspondence of Vignéras and its relationship with the usual local Langlands correspondence via reduction modulo $\ell$.
In \S \ref{section: equivariant etale cohomology} we review certain homotopical complexes representing étale cohomology with group actions. These were constructed by Rickard \cite{Ric94} and give an important tool for our comparison of characteristic zero and characteristic $\ell$ cohomology.

In \S \ref{section: p-adic deligne--lusztig theory} we turn towards $p$-adic Deligne--Lusztig theory, concentrating on the hyperspecial Deligne--Lusztig space of Coxeter type for $\mathbf{GL}_n$. We generalize the construction of the isotypic parts in the étale cohomology groups $H^\bullet_c(\dot{X}, \Lambda)$ of \cite{Iva18, Iva19, Iva20a, Iva20b} to the finer setting of the bounded derived category $\D^b(G, \Lambda)$ for $\Lambda$ such as $\Zl$ and $\Fl$. At the same time, we record a few statements from \cite{Iva18,Iva19,Iva20a,Iva20b} about the characteristic zero case and employ them to get control over its characteristic $\ell$ counterpart.
We then use the body of this paper to establish the necessary compatibilities between the vertical maps and the horizontal reduction modulo $\ell$ in Diagram \ref{diagram}, deducing Theorem \ref{theorem: upshot}.
We finally comment on a non-derived version of our result in Theorem \ref{theorem: naive upshot}, leading to extra multiplicities for those $\ell$ which are small with respect to $\mathbf{G}$. This yields extra information about the structure of the cohomology.

\subsection{Notation}\label{notation: non-archimedean fields}
Throughout, $K$ will be a local non-archimedean field with ring of integers $\O_K$, maximal ideal $\mathfrak{p}_K$ and residue field $k = \F_q$ of size $q$ and characteristic $p$. 
A uniformizer of $\O_K$ is denoted $\varpi = \varpi_K$. Such $K$ is either a finite extension of the $p$-adic numbers $\Q_p$ with residue field $\F_q$, or the function field $\F_q(\!(\varpi)\!)$. 

We fix an integer $n \in \N$ and denote by $L$ the degree $n$ unramified extension of $K$. Thus the residue field of $L$ is $\F_{q^n}$. 
We fix a different prime $\ell\neq p$ and let $m = v_{\ell}(q^n-1)$ be the $\ell$-adic valuation of $q^n-1 =|\F^\times_{q^n}|$. In other words, $\ell^m$ is the size of the $\ell$-torsion part of $\F^\times_{q^n}$.
 
\subsection{Acknowledgements}
This paper originated from my master thesis written at the University of Bonn during the academic year 2020/2021. I want to thank Alexander Ivanov for mathematical discussions and feedback on earlier drafts. I further want to thank Peter Scholze for suggestions about local systems, leading to a clarification of the result for small $\ell$. 

During my stay at the University of Bonn, I was supported by the DAAD study scholarship for graduates of all disciplines (57440925).
Some final improvements of this paper were done at the Institute of Science and Technology Austria.
\section{Grothendieck groups of smooth representations}
\label{section: grothendieck groups of smooth representations}

We review the representation theory of locally profinite groups in sufficient generality valid also for characteristic $\ell$ coefficients. A very good reference is the book \cite{Vig96}. See also \cite{BH06} for coefficients of characteristic zero.

\subsection{Coefficients rings}\label{section: coefficient rings}
We will be considering representations with coefficients in a commutative ring $\Lambda$. The interest lies in the following choices.

\begin{setup}\label{setting: coefficient rings}
Fix a prime $\ell \neq p$. Let $\Q_\ell$, $\Z_\ell \supseteq \mathfrak{m}$, $\F_\ell$ denote respectively the $\ell$-adic numbers, their ring of integers with its maximal ideal, and their residue field of size $\ell$. 
By $\overline{\Q}_\ell$, $\overline{\Z}_\ell \supseteq \overline{\mathfrak{m}}$, $\overline{\F}_\ell$ we mean the algebraic closure of $\Q_\ell$ with its ring of integers, its maximal ideal, and residue field. 
In general, we write $E$, $\O_E \supseteq \mathfrak{m}_E$, $F$ for a field extension $\Q_\ell \subseteq E  \subseteq \overline{\Q}_\ell$ with ring of integers $\O_E$, its maximal ideal $\mathfrak{m}_E$, and residue field $F$. For finite extensions, a uniformizer of $\O_E$ is denoted $\varpi_E$.
\end{setup}

Denote respectively $\mu_r$, $\mu_{\infty}$, $\mu_{\ell^{\infty}}$, $\mu_{\ell^{\infty}}^{\perp}$ the functors of $r$-th roots of unity, roots of unity, roots of unity of order a power of $\ell$, roots of unity of order coprime to $\ell$. 
The canonical surjection $\O_E \to F$ induces a multiplicative map $\O_E^{\times} \rightarrow F^{\times}$. The {\it Teichmüller map} gives a canonical multiplicative section $\O_E^{\times} \leftarrow F^{\times}$. It identifies $F^\times = \mu_{\infty}(F)$ with the direct summand $\mu_{\ell^{\infty}}^{\perp}(E)$ of the multiplicative group $E^\times$ whose torsion part is $\mu_{\infty}(E) =\mu_{\ell^{\infty}}(E) \times \mu_{\ell^{\infty}}^{\perp}(E)$.

\subsection{Categories of smooth representations and their Grothendieck groups}
\label{subsection: categories of smooth representations and their grothendieck groups}

Given a locally profinite topological group $G$ and a coefficient ring $\Lambda$, we have the category of {\it smooth} representations $\Rep(G, \Lambda)$ of $G$ on $\Lambda$-modules, i.e. continuous representations with respect to the discrete topology on $\Lambda$-modules. We denote by $\K^b(G, \Lambda)$ and $\D^b(G, \Lambda)$ its bounded homotopy category and bounded derived category.

A closed subgroup $H \leq G$ is again locally profinite and we have the functors of compact induction $\cInd^G_H$, restriction $\Res^G_H$ and induction $\Ind^G_H$.
When $G$ possesses a compact open subgroup whose pro-order is invertible in $\Lambda$, all three of the above functors are exact \cite[\S I.5.10]{Vig96}. 

\begin{remark}\label{remark: clopeness}
Any open subgroup $H \leq G$ is automatically closed (hence clopen). Indeed, $H$ is the complement of its other cosets $gH$, which are open as images of $H$ under the homeomorphisms $g\cdot: G \to G$ for varying $g \in G$.
\end{remark}

If $H \leq G$ is clopen, $\cInd^G_H$, $\Res^G_H$ and $\Ind^G_H$ form two adjoint pairs in this ordering. Moreover, $\cInd^G_H$ is exact: for each $V \in \Rep(H, \Lambda)$, the underlying vector space of $\cInd^G_H V$ is given by $\bigoplus_{H\backslash G} V$. 
See \cite[\S I.5]{Vig96} for details.

\begin{definition}
A representation $V \in \Rep(G, \Lambda)$ is called 
\begin{itemize}
    \item of {\it finite length}, if it has a finite composition series,
    \item of {\it finite type}, if it is finitely generated as $G$-representation,
    \item  {\it admissible}, if $V^H$ is a finitely generated $\Lambda$-module for each compact open subgroup $H \leq G$.
\end{itemize}
\end{definition}
\noindent In order to have well-behaved Grothendieck groups, we will further restrict to the subcategory
\begin{equation*}
\rep(G, \Lambda) \subseteq \Rep(G, \Lambda)    
\end{equation*}
of representations of finite length.\footnote{This is necessary so that the Grothendieck group $\G_0(\rep(G, \Lambda))$ does not degenerate by Eilenberg swindle.}
We denote by $\Rep(G, \Lambda)^{\ft} \subseteq \Rep(G, \Lambda)$ the subcategory of representations of finite type.
A representation of finite length is automatically of finite type.

\begin{example}
Assume $\Lambda$ is artinian. Note that if $G$ is finite, $\rep(G, \Lambda) \cong \mo(\Lambda[G])$ is the category of finitely generated representations of $G$ over $\Lambda$. 
Similarly when $G$ is profinite, the conditions ``finite type" and ``finite length" are equivalent for smooth representations, so $\rep(G, \Lambda)$ is the subcategory of representations of finite type. 
\end{example}

\begin{lemma}\label{lemma: induction and finite type}
Let $\Lambda$ be a coefficient ring and $H \leq G$ a clopen subgroup. Then $\cInd^G_H$ preserves the property of being of finite type.
\end{lemma}

\begin{notation}
Given a coefficient ring $\Lambda$ and a locally profinite group $G$, we denote by $\G_0(G, \Lambda)$ the Grothendieck group of the category $\rep(G, \Lambda)$ of smooth representations of $G$ over $\Lambda$ of finite length discussed above.

We denote by $\Irr(G, \Lambda)$ a set of representatives of isomorphism classes of irreducibles in $\rep(G, \Lambda)$.
The imposed finite length condition forces $\G_0(G, \Lambda)$ to be the free abelian group on $\Irr(G, \Lambda)$.
\end{notation}
 
When $G$ is abelian and $\Lambda$ is an algebraically closed field (and $G/G'$ is countable for any compact open subgroup $G'$), the Schur lemma holds: all elements of $\Irr(G, \Lambda)$ are one-dimensional. For more general $\Lambda$, we write $\Hom_{\Grp}(G, \Lambda^\times)$ for the character group.

\subsection{Integral representations}
Consider a coefficient ring $\Lambda$ among $E, \O_E, F$ as in Setup \ref{setting: coefficient rings}. In order to relate these different choices of coefficients, we recall the subcategory $\rep(G, E)^{\integral} \subseteq \rep(G, E)$ of {\it integral representations}.

\begin{definition}\label{integral}
A representation $V \in \Rep(G, E)$ is called {\it integral} if it lies in the essential image of the functor $(E \otimes_{\O_E} -): \Rep(G, \O_E) \to \Rep(G, E)$. 
\end{definition}
In other words, $V$ is integral if it contains a full $G$-stable $\O_E$-lattice $M$. A choice of such isomorphism is referred to as an {\it integral structure} of $V$.

\begin{definition}\label{integralGrot}
Let $\rep(G, E)^{\integral} \subseteq \rep(G, E)$ denote the full abelian subcategory of integral representations. 
We denote by $\G_0(G, E)^{\integral} \leq \G_0(G, E)$ the subgroup given by the Grothendieck group of $\rep(G, E)^{\integral}$.
\end{definition}

Notice that there is a small statement to check, namely that integral representations cut out an abelian subcategory closed on subobjects, quotients and extensions -- we then really have a map on Grothendieck groups, which is injective by Jordan--Hölder theorem.
\begin{lemma}\label{lemma: integral ses}
Let $0 \to V' \to V \to  V'' \to 0$ be a short exact sequence in $\rep(G, E)$. Then $V$ is integral $\iff$ both $V'$ and $V''$ are integral.
\end{lemma}

\begin{proof}
Given an integral structure $M$ on a representation $V$, we obtain an integral structure on any subrepresentation resp. quotient by intersecting resp. quotienting $M$. On the other hand, if $V'$, $V''$ have integral structures $M'$, $M''$, we obtain an integral structure $M$ on their extension $0 \to V' \to V \to V'' \to 0$ by lifting $M''$ to some $\widetilde{M}''$, rescaling $M'$ by some $c \in E^{\times}$ so that $G$ maps some finite generating set of $\widetilde{M}'' \in \rep(G, \O_E)$ into $cM' + \widetilde{M}''$, and finally putting $M:= cM' + \widetilde{M}''$.
\end{proof}

\begin{example}
If $G$ is profinite, an integral structure for $V \in \rep(G, E)$ always exists: one takes any finite generating set of $V$ and closes it on translations by $G$. The resulting set is still finite, hence generates an $\O_E$-lattice $M$ inside $V$.
\end{example}

\begin{remark}\label{example:Z}
Note that for locally profinite $G$ an integral structure need not exist.
The basic counterexample appears for characters: A character $\theta: G \to E^\times$  is integral if and only if it factors through the inclusion $\O_E^{\times} \subseteq E^{\times}$.

Indeed -- up to a scalar, there is only one candidate for an integral structure $M$ of the one-dimensional representation given by $\theta$. This $M$ is stable if and only if $\theta$ doesn't hit any element with negative $\ell$-adic valuation, i.e. if and only if $\theta$ hits only elements of $E^{\times}$ with zero $\ell$-adic valuation.
\end{remark}

\begin{example}\label{example:T}
A relevant example of the above failure appears for the multiplicative group $G:= K^{\times}$ of a non-archimedean local field $K$. Any smooth character of $G$ is then uniquely determined by its restriction to a smooth character of the profinite group $\O_K^{\times}$ and its value at some uniformizer $\varpi_K$. 
This gives an identification of abelian groups
$$\Hom_{\Grp}(K^\times, E^\times) \cong \Hom_{\Grp}(\O_K^{\times}, E^\times) \times E^\times.$$

Since each smooth character $\theta$ of $\O_K^{\times}$ factors through a finite quotient, its image lies in the torsion subgroup $\mu(E)$ whose elements have zero $\ell$-adic valuation. 
The integrality of $\theta$ thus depends only on the $\ell$-adic valuation of the image of the uniformizer. The subgroup of integral characters thus identifies as $\Hom_{\Grp}(K^{\times}, E^{\times})^{\integral} = \Hom_{\Grp}(\O_K^{\times}, E^\times) \times \O_{E}^\times \subseteq \Hom_{\Grp}(\O_K^{\times}, E^\times) \times E^\times$.
\end{example}

\begin{example}
For $p$-adic reductive groups, the above examples describe the essence of the failure of integrality -- by \cite[II.4.13]{Vig96}, a representation of such $G$ over $\overline{\Q}_\ell$ is integral if and only if the central character of its cuspidal support is integral.

For $G = GL_n(K)$, we have $Z(GL_n(K)) = K^\times$ via the diagonal embedding. The integrality of a cuspidal representation of $GL_n(K)$ is thus equivalent to the integrality of the corresponding character of $K^\times$, discussed in the previous example.
\end{example}

\begin{lemma} \label{lemma:induction and integrality}
Let $G$ be a locally profinite group and $H$ an open subgroup. If a smooth representation $V$ of $H$ is integral, then so is $\cInd_H^G V$.
\end{lemma}
\begin{proof}
Let $V \in \Rep(H, E)$, and let $M$ be an integral structure of $V$. We claim that $\cInd ^G_{H} M$ gives an integral structure of $\cInd ^G_{H} V$. The natural map $E \otimes_{\O_E} \cInd ^G_{H} M \to \cInd ^G_{H} V$ is injective because $M$ is a lattice in $V$. 

On the other hand, any $f \in \cInd ^G_{H} V$ is compactly supported modulo $H$. Hence it is supported on finitely many points of the discrete quotient space $H\backslash G$, this being the case by clopeness of $H$. Because $M$ is a full lattice, we can now multiply $f$ by a suitable scalar so that its values for each of these finitely many cosets land in $M$ -- indeed, $M$ is stable under the $H$-action, so all values of $f$ on a given coset lie in $M$ whenever one of them does. This shows surjectivity, proving the lemma.
\end{proof}
In other words, $\cInd^G_{H}$ canonically preserves integral structures on smooth representations.

\subsection{Reduction modulo \texorpdfstring{$\ell$}{l}}\label{reductionmodulol}
Consider a coefficient ring $\Lambda$ among $E, \O_E, F$ as in Setup \ref{setting: coefficient rings}. For a locally profinite group $G$, we will discuss the existence of {\it reduction modulo $\ell$} map
$$\G_0(G, E) \geq \G_0(G, E)^{\integral} \xrightarrow{r_\ell} \G_0(G, F).$$

This $r_\ell$ is only well-defined on the level of Grothendieck groups of finite length representations. The domain $\G_0(G, E)^{\integral}$ of $r_\ell$ was introduced in Definition \ref{integralGrot}.

When $G$ is finite, $r_\ell$ forms one side of the so-called Cartan--Brauer triangle; it is usually called the {\it decomposition map} \cite{Sch}. 
With certain technical restrictions, $r_\ell$ is defined also in the locally profinite setting \cite{Vig96}. 
We now describe the construction of $r_\ell$; the well-definedness in cases of interest is discussed immediately afterwards.

\begin{construction}\label{constr}
Start with $V \in \rep(G, E)^{\integral}$ and choose an integral structure $M$. Then $M/\mathfrak{m}M$ is naturally a $G$-representation with coefficients in the residue field $F$. If $M/\mathfrak{m}M$ has finite length, it lies in $\rep(G, F)$ and we may pass to its class $r_\ell(V) := [V] \in \G_0(G, F)$. 
\end{construction}

\begin{remark}\label{remark: reduction on subcategories}
Whenever Construction \ref{constr} determines a well-defined map $\mathcal{A} \to \G_0(G, F)$ independent of the choices of integral structures on some abelian subcategory $\mathcal{A}$ of $\rep(G, E)^{\integral}$, it automatically gives a map 
$\G_0(\mathcal{A}) \to \G_0(G, F)$ on Grothendieck groups. We will denote this map by $r_\ell$ as well.

To see this, one only needs to note that relations given by short exact sequences are sent to such relations again. Given a short exact sequence $0 \to V' \to V \to V'' \to 0$ in $\mathcal{A}$, one can choose integral structures fitting into a short exact sequence $0 \to M' \to M \to M'' \to 0$ as in the proof of Lemma \ref{lemma: integral ses} and reduce via these, obtaining the desired relation. 
\end{remark}

\begin{remark}
Note that $r_\ell$ preserves vector space dimension.
\end{remark}

The following {\it Brauer-Nesbitt principle} addresses the independence of $r_\ell$ on the choice of integral structure in reasonable generality.

\begin{proposition}[{\cite[\S I.9.6]{Vig96}}]\label{BrauerNesbitt}
Let $G$ be a locally profinite group. Assume $V \in \rep(G, E)$ is admissible and contains an integral structure $M$ of finite type such that $M/\mathfrak{m}M$ has finite length.
Then $r_\ell(V) \in \G_0(G, F)$ is well-defined and independent of the choice of $M$.
\end{proposition}

\begin{discussion}\label{reductionmodulolitems}
Let us now record three important instances when Construction \ref{constr} works. These include all cases relevant later. 
\begin{itemize}
\item[(i)] If $G$ is profinite, the above construction yields a well-defined map
$$r_\ell: \G_0(G, E) \to \G_0(G, F).$$
Indeed, all representations $V$ in question are automatically finite-dimensional by compactness of $G$. We have already observed this implies $\rep(G, E)^{\integral} = \rep(G, E)$ for profinite $G$. The assumptions on admissibility and finite length of Proposition \ref{BrauerNesbitt} are clearly satisfied for finite dimensional representations, so the reduction $r_\ell(V) \in \G_0(G, F)$ is independent of the choices made. This is a priori true for irreducible $V$, but immediately extends to all representations by looking at their composition series. By Remark \ref{remark: reduction on subcategories}, we get a map $\G_0(G, E) \to \G_0(G, F)$. 

\item[(ii)] If $G$ is a locally profinite group and $\G_0(G, E)^{\fd, \: \integral}$ is the Grothendieck group of the category $\rep(G, E)^{\fd, \: \integral}$ of finite-dimensional integral representations, the above construction yields a well-defined map
$$r_\ell: \G_0(G, E)^{\fd, \: \integral} \to \G_0(G, F)^{\fd}.$$  

The argument is the same as in (i), the only difference being that the finite-dimensionality comes as an assumption. 

\item[(iii)] If $G$ is a $p$-adic reductive group, the above construction yields a well-defined map
$$r_\ell: \G_0(G, E)^{\integral} \to \G_0(G, F).$$

Such $G$ is locally profinite, but we are now dealing with infinite-dimensional representations. Luckily, the assumptions of Proposition \ref{BrauerNesbitt} (for irreducible $V$) hold by \cite[II.5.11.b]{Vig96}. Note that the admissibility condition in this proposition is automatic, since any irreducible representation (and hence any finite-length representation) of $G$ is admissible. The extension to all finite length representation and to a map of Grothendieck groups is again immediate.
\end{itemize}
\end{discussion}

\begin{example}\label{finfib}
\label{remark: fibers of the reduction for abelian groups}
Take $G$ finite abelian and $\Lambda = \Ql$, $\Fl$.
Then each irreducible representation is a character and the reduction $r_\ell: \G_0(G, \Ql) \to \G_0(G, \Fl)$ is induced by postcomposing characters with $\mu_{\infty}(\Ql) = \mu_{\infty}(\Zl) \to \mu_{\infty}(\Fl)$.
Since $\mu_{\infty}(\Fl)$ sits inside $\mu_{\infty}(\Ql)$ as the direct summand $\mu_{\ell^\infty}^\perp(\Ql)$, the fibers of the surjection $\Irr(G, \overline{\Q}_\ell) \to \Irr(G, \overline{\F}_\ell)$ are given by $\Hom_{\Grp}(G, \mu_{\ell^\infty}(\Ql))$. In particular, the size of these fibers is given by the number of elements of $G$ whose order is a power of $\ell$.
\end{example}

\subsection{Compatibility of \texorpdfstring{$r_{\ell}$}{rl} and \texorpdfstring{$\cInd$}{cInd}}
Let $G$ be a locally profinite group and $H \leq G$ a clopen subgroup. 
The functor $\cInd_{H}^G: \Rep(H, \Lambda) \to \Rep(G, \Lambda) $ is exact by \S \ref{subsection: categories of smooth representations and their grothendieck groups}, but it need not restrict to a map $\rep(H, \Lambda) \rightarrow \rep(G, \Lambda)$ between finite length representations nor the associated Grothendieck groups.

\begin{notation}\label{notation: subcategories of representations}
We denote by $\rep(G, \Lambda)^{\subcat} \subseteq \rep(G, \Lambda)$ any Serre subcategory: a nonempty strictly full (abelian) subcategory closed on subobjects, quotients and extensions \cite[Tag 02MP]{Sta}. By looking at Jordan-Hölder series, such $\rep(G, \Lambda)^{\subcat}$ corresponds to a subset of irreducible objects in $\rep(G, \Lambda)$. Its Grothendieck group $\G_0(G, \Lambda)^{\subcat} \leq \G_0(G, \Lambda)$ is the subgroup spanned by these irreducibles.
\end{notation}

If $\cInd_{H}^G$ restricts to a functor $\rep(H, \Lambda)^{\subcat} \to \rep(G, \Lambda)^{\subcat}$, we get an induced map $\cInd_{H}^G: \G_0(H, \Lambda)^{\subcat} \to \G_0(G, \Lambda)^{\subcat}$. On classes of genuine representations, this sends $[V] \mapsto [\cInd_H^G(V)]$. We now discuss its behaviour with respect to $r_\ell$.

\begin{lemma}\label{lemma: cind and rl}
Let $G$ be a locally profinite group and $H \leq G$ a clopen subgroup. Consider subcategories as in Notation \ref{notation: subcategories of representations} such that the following two reduction maps $r_{\ell}$ are well-defined:
\begin{equation*}\label{equation: cind and rl}
\begin{tikzcd}
 \G_0(G, E)^{\subcat,  \integral} \arrow[r, "r_{\ell}"]& \G_0(G, F), &  &
 \G_0(H, E)^{\subcat, \integral} \arrow[r, "r_{\ell}"]& \G_0(H, F) 
\end{tikzcd}
\end{equation*}
Let $[V] \in \G_0(H, E)^{\subcat, \integral}$ be a class of a genuine representation $V$ such that $\cInd_H^G[V]$ lies in $\G_0(G, E)^{\subcat}$. Then it is integral and
\begin{equation*}
    r_{\ell} \circ \cInd_H^G [V] = \cInd_H^G \circ \ r_{\ell} [V] \in \G_0(G, F).
\end{equation*}
\end{lemma}
\begin{proof}
First note that $\cInd_H^G[V] \in \G_0(G, E)^{\subcat, \integral}$.
Indeed, it has finite length by assumption and it is integral by Lemma \ref{lemma:induction and integrality} -- given $V \in \rep(H, E)$ with an integral structure $M$, the compact induction $\cInd ^G_{H} M$ gives an integral structure of $\cInd ^G_{H} V$.

Reducing by $r_\ell$ with respect to these two integral structures shows that
\begin{equation*}
    r_{\ell} \circ \cInd_H^G [V] = \cInd_H^G \circ \ r_{\ell} [V],
\end{equation*}
because $\cInd ^G_{H}$ over $\O_E$-coefficients commutes with modding out the image of the maximal ideal $\mathfrak{m}_E \subseteq \O_E$ by explicit inspection. 
\end{proof}

\begin{remark}\label{remark: cind and rl}
In particular, the setup of Lemma \ref{lemma: cind and rl} implies that $\cInd_H^G \circ \ r_{\ell} [V]$ is a well-defined class in the Grothendieck group of finite length representations $\G_0(G, F)$.
\end{remark}

\begin{corollary}\label{corollary: cind and rl square}
Suppose we are given subcategories of finite length representations as in Notation \ref{notation: subcategories of representations} so that the following square is well-defined. Then it commutes.
\begin{center}
\begin{tikzcd}
\G_0(G, E)^{\subcat, \integral} \arrow[r, "r_\ell"] &  \G_0(G, F)^{\subcat}   \\
\G_0(H, E)^{\subcat, \integral} \arrow[r, "r_\ell"] \arrow[u, "\cInd ^G_{H}"] &  \G_0(H, F)^{\subcat}  \arrow[u, "\cInd ^G_{H}"]
\end{tikzcd}
\end{center}
\end{corollary}
\begin{proof}
Use Lemma \ref{lemma: cind and rl} for classes of irreducible representations $[V] \in \G_0(H, E)^{\subcat, \integral}$.     
\end{proof}

\subsection{Finite groups and permutation modules}

For a finite group $G$ and a commutative coefficient ring $\Lambda$, we have the subclass of objects in $\Rep(G, \Lambda)^{\ft}$ of {\it permutation modules}. By definition, these are cut out by the essential image of the free $\Lambda$-module functor 
$$\Lambda[-]: \FinSet_G \to \Rep(G, \Lambda)^{\ft} \qquad \text{sending} \qquad S \mapsto \Lambda[S].$$
For each $S \in \FinSet_G$, denote $\varepsilon: \Lambda[S] \to \Lambda$ the natural augmentation map given by summing up the coefficients in the canonical basis of $\Lambda[S]$.

We recall the following standard characterization of projective permutation modules.
\begin{lemma}\label{lemma: projectivity of permutation modules}
Let $H \leq G$ be a subgroup. Then the permutation module $\Lambda[G/H] \in \Rep(G, \Lambda)^{\ft}$ is projective if and only if $|H| \in \Lambda^\times$.
\end{lemma}
\begin{proof}
The projectivity of $\Lambda[G/H]$ is equivalent to the splitting of the projection $\pi: \Lambda[G] \twoheadrightarrow \Lambda[G/H]$. 
To give a splitting of $\pi$ is the same thing as to lift the element $1 \cdot H \in \Lambda[G/H]$ to some $h \in \Lambda[G]$ invariant under the left action of $H$. Writing such $h$ in the natural basis labelled by elements of $G$, it must have constant coefficients along the left cosets of $H$.

If $|H| \notin \Lambda^\times$, such lift $h$ does not exist as $\varepsilon(h) \notin \Lambda^\times$ but $\varepsilon(1\cdot H) = 1$. On the other hand if $|H| \in \Lambda^\times$, the element $h := |H|^{-1} \cdot \sum_{g \in H} g$ gives such lift.
\end{proof}

\begin{setup}\label{setting: noncommutative coefficient rings}
We will use the following (possibly non-commutative) rings $A$ as coefficients. All modules are implicitly left modules.
\begin{itemize}
    \item[(i)] We denote by $A$ any torsion Artin algebra. In particular, this covers the case $A = \Lambda[G]$ with $\Lambda = F$ or its finite self-extensions.
    \item[(ii)] We further allow $A = \Lambda[G]$ for any $\Lambda$ from Setup \ref{setting: coefficient rings}. Such $A$ is an inverse limit of a flat system of torsion Artin algebras, or its subsequent flat base change.
\end{itemize}
\end{setup}

For a subclass $\mathcal{M}$ of objects of an abelian category $\mathcal{A}$, we write $\add(\mathcal{M})$ for the smallest additive idempotent-complete full subcategory of $\mathcal{A}$ containing $\mathcal{M}$.
In particular, starting from the class $\mathcal{M} = \{ \Lambda[G/H] \mid \ell \nmid |H| \}$ inside $\rep(G, \Lambda)$, we see that $\add (\mathcal{M}) = \add(\Lambda[G])$ consists precisely of finite projective $\Lambda[G]$-modules.
Indeed, since $\mathcal{M}$ contains only projectives by Lemma \ref{lemma: projectivity of permutation modules}, the same is true for $\add(\mathcal{M})$; the other direction is obvious as $\Lambda[G] \in \mathcal{M}$.
\section{The structure of \texorpdfstring{$T$}{T} and its characters} \label{structureofT}
\label{section: the structure of T and its characters}

In this section we discuss the characters of $T = L^\times$ with values in $\Lambda = \overline{\Q}_\ell$, $\overline{\F}_\ell$. Apart from their general structure, we consider the question of their integrality and discuss the reduction map $r_\ell$. We keep track of the Galois action on them.

\subsection{Characters of \texorpdfstring{$T$}{T}}

Let $T = L^\times$ be the locally profinite group of units in the degree $n$ unramified extension $L/K$ as in \S \ref{notation: non-archimedean fields}. We denote by $T_{\O} = T^0 = \O_L^\times$ the units in the ring of integers $\O_L$ of $L$; this is a profinite subgroup of $T$. 
Note that $L^\times$ can be written as the pushout of abelian groups
\begin{equation}\label{diagram:field pushout}
\begin{tikzcd}
K^\times \arrow[hookrightarrow]{r}   & L^\times \\
\O_K^\times \arrow[hookrightarrow]{r}  \arrow[hookrightarrow]{u}  & \O_L^\times \arrow[hookrightarrow]{u}
\end{tikzcd}
\end{equation}
Thus for $\Lambda = \overline{\Q}_\ell, \overline{\F}_\ell$ we can write
\begin{equation}\label{equation: characters of T}
\Irr(T, \Lambda) =  \Irr(K^\times, \Lambda) \times_{\Irr(\O_K^\times, \Lambda)} \Irr(\O_L^\times, \Lambda) = \Lambda^\times \times \Irr(T_{\O}, \Lambda),    
\end{equation}
the second equality given by fixing a uniformizer $\varpi_K$.
For $\Lambda = \overline{\Q}_\ell$, the question of integrality is answered as in Example \ref{example:T} -- a character is integral if and only if the image $\varpi_K$ lies in $\overline{\Z}_\ell^\times$.

\begin{lemma}\label{lemma: lifting characters}
The reduction map $r_\ell: \Irr(T, \overline{\Q}_\ell)^{\integral} \xrightarrow{} \Irr(T, \overline{\F}_\ell)$ is surjective.
\end{lemma} 
\begin{proof}
Postcomposition with the Teichmüller lift gives a section of $\Irr(T, \overline{\Q}_\ell) \xrightarrow{} \Irr(T, \overline{\F}_\ell)$. Since the image of the Teichmüller map lands in $\mu_{\infty}(\overline{\Q}_\ell) \subseteq \overline{\Z}_\ell^\times$, this section factors though $\Irr(T, \overline{\Q}_\ell)^{\integral}$.
\end{proof}

Now we return to the square \eqref{diagram:field pushout} to address the Galois action. The group $\Gal(L/K)$ acts on $\Irr(T, \Lambda)$ by precomposition with its defining action. 
Under the identification \eqref{equation: characters of T} this action decomposes into the trivial action (induced by precomposition on $K^\times$) times the natural action on $\O_L^\times$.
Altogether, there are two independent contributions to $\Irr(T, \Lambda)$ -- the ``profinite" contribution of $T_{\O}$ and the ``free" contribution of the uniformizer.

\subsection{Characters of the profinite \texorpdfstring{$T_{\O}$}{TO}}\label{Tsec -- prof}

We have the locally profinite group $T = L^{\times}$ with its  profinite subgroup $T_{\O}  = T^0 = \O_L^\times$. For $h \geq 1$ set $T^h := 1 + \mathfrak{p}_L^h$.
For any $0 \leq h \leq h'$, we obtain the finite group $T^{h}_{h'} := T^{h}/T^{h'}$. In particular, we get the truncations $T_h := T^0_h = \O^\times_L / T^{h}_L$.

The profinite group $\O_L^\times$ has a decreasing filtration 
$$\O_L^{\times} = T^0 \supset T^1 \supset \dots \supset T^h \supset \dots$$
by clopen subgroups which form a basis of neighbourhoods of $1$; this filtration is also $\Gal(L/K)$-stable. 
The truncation $T_h$ fits into a $\Gal(L/K)$-equivariant short exact sequence of abelian groups
\begin{equation}\label{ses}
1 \to T^1_h \to T_h \to T_1 \to 1.
\end{equation}
The cokernel $T_1$ is isomorphic to ${\F^{\times}_{q^n}}$. On the other hand, the kernel $T^1_h$ has size $q^{n(h-1)}$. So the short exact sequence \eqref{ses} is canonically split, giving the decomposition of $T_h$ into $p$-torsion part and $p$-torsion-free part. As $\ell \neq p$, the size of the kernel $T^1_h$ is coprime to $\ell$.

Applying $\Hom_{\Grp}(-, \Lambda^\times)$ to the split sequence \eqref{ses} gives a split short exact sequence of character groups with coefficients in $\Lambda$, naturally in $\Lambda$.
In particular, we apply this for $ \Lambda = \overline{\Q}_\ell, \overline{\Z}_\ell, \overline{\F}_\ell$. 
Comparing characters with values in $\Ql, \Fl$ via the natural maps from characters with values in $\overline{\Z}_\ell$ implies that this sequence is compatible with reduction modulo $\ell$ -- the following diagram is commutative and $\Gal(L/K)$-equivariant.
\begin{center}
\begin{tikzcd}
1  & \Irr (T^1_h, \overline{\Q}_\ell) \arrow{l} \arrow{d}{r_\ell} & \Irr(T_h, \overline{\Q}_\ell) \arrow{l} \arrow{d}{r_\ell} & \Irr(T_1, \overline{\Q}_\ell) \arrow{l}  \arrow{d}{r_\ell} & \arrow{l}  1 \\ 
1 & \Irr(T^1_h, \overline{\F}_\ell) \arrow{l}  & \Irr(T_h, \overline{\F}_\ell) \arrow{l} & \Irr(T_1, \overline{\F}_\ell) \arrow{l} & \arrow{l} 1 
\end{tikzcd}
\end{center}

Since the order of $T^1_h$ is coprime to $\ell$, the left-hand side reduction map is an isomorphism by Example \ref{remark: fibers of the reduction for abelian groups}. On the other hand, the order of $T_1$ need not be coprime to $\ell$. If we denote by $m$ its $\ell$-adic valuation, it also follows from Example \ref{remark: fibers of the reduction for abelian groups} that the right-hand map is a surjection with fibers of size $\ell^{m}$. We may record the following.

\begin{lemma}\label{Tlifts}
\label{lemma: counting lifting characters}
The reduction $r_\ell: \Irr(T_h, \overline{\Q}_\ell) \xrightarrow{} \Irr(T_h, \overline{\F}_\ell)$ is a $\Gal(L/K)$-equivariant surjection with fibers of size $\ell^m$.
\end{lemma}

The smoothness of $\theta \in \Irr(T_{\O}, \Lambda)$ is equivalent to the existence of an integer $h \in \N$ such that $\theta$ factors over $T_h$. The smallest such $h$ is called the {\it level} of $\theta$ and denoted $\lvl(\theta)$. We will often confuse $\theta \in \Irr(T_{\O}, \Lambda)$ with such factorization in $\Irr(T_h, \Lambda)$ for some $h \geq \lvl(\theta)$.
The statement of Lemma \ref{lemma: counting lifting characters} carries word for word to a statement about smooth characters of the profinite group $T_{\O}$.

\begin{remark}\label{remark: integral lifts of characters}
Altogether, given a character $\psi \in \Irr(T, \overline{\F}_\ell)$, we can choose lifts $\theta_i \in \Irr(T, \overline{\Q}_\ell)^{\integral}$ for $i = 1, \dots, \ell^m$ of $\psi$ under $r_\ell$ whose restrictions to $T_{\O}$ are precisely the unique $\ell^m$ lifts of $\psi|_{T_{\O}}$. Indeed, one simply needs to lift the image of the uniformizer $\varpi \in {\O_K}$ from $\overline{\F}_\ell^{\times}$ to $\overline{\Z}_\ell^{\times}$ as discussed in the previous section -- the Teichmüller map provides a canonical choice.
\end{remark}

\subsection{General position of characters}\label{Tsec -- gen}
We now discuss the action of $\Gal(L/K)$ on $\Irr(T, \Lambda)$ by precomposition more closely. There is the following standard vocabulary.

\begin{definition}\label{generalposition}
\label{definition: general position}
A smooth character $\theta \in \Irr(T, \Lambda)$ lies in {\it general position} if it has trivial $\Gal(L/K)$-stabilizer.
We denote the set of characters in general position by $\Irr(T, \Lambda)^{\gen}$. When referring to the local Langlands correspondence, it is also denoted $\mathscr{X}(L, \Lambda)$.
\end{definition}

\begin{remark}
Since $T = L^\times$ is the pushout \eqref{diagram:field pushout}, the $\Gal(L/K)$-action on $T$ is uniquely determined by its restriction to $T_{\O}$.  
Therefore, a character $\theta \in \Irr(T, \Lambda)$ is in general position if and only if its restriction $\theta|_{T_{\O}}$ has trivial $\Gal(L/K)$-stabilizer.
\end{remark}

The discussion of the preceding section shows that for our purposes, $\theta|_{T^1}$ is the better behaved part of $\theta$. We consider the following stronger condition.
\begin{definition}\label{stronglygeneralposition}
\label{definition: strongly general position}
A smooth character $\theta \in \Irr(T, \Lambda)$ lies in {\it strongly general position} if $\theta |_{T^1}$ has trivial stabilizer in $\Gal(L/K)$.
We denote the subset of characters in strongly general position by $\Irr(T, \Lambda)^{\sg}$ or by $\mathscr{X}(L, \Lambda)^{\sg}$.
\end{definition}
One reason for this definition is the appearance of the same condition in \cite[Theorem A]{Iva19}, which we use crucially. Another reason is that it behaves well with respect to $r_\ell$.

\begin{lemma}\label{Tposition}
\label{lemma: strongly general position and reduction}
A smooth character $\theta \in \Irr(T, \overline{\Q}_\ell)^{\integral}$ lies in strongly general position if and only if its reduction $r_\ell(\theta) \in \Irr(T, \overline{\F}_\ell)$ lies in strongly general position.
\end{lemma}
\begin{proof}
By smoothness, $\theta|_{T_{\O}}$ is a character of $T_h$ for some big enough $h$. Lying in strongly general position means that the restriction to $T^1_h$ has trivial $\Gal(L/K)$-stabilizer. But the reduction map $r_\ell: \Irr(T^1_h, \overline{\Q}_\ell) \xrightarrow{} \Irr(T^1_h, \overline{\F}_\ell)$ is a $\Gal(L/K)$-equivariant isomorphism, implying the result.
\end{proof}

\begin{remark}
Note that Lemma \ref{Tposition} would not be true if we replaced {\it strongly general position} by {\it general position} everywhere.
\end{remark}
\section{Local Langlands correspondence}
\label{section: local langlands correspondence}

In this section we recall certain facts about the local Langlands correspondence for $G = \mathbf{GL}_n(K)$ with coefficients $\Lambda = \overline{\Q}_\ell$, $\overline{\F}_\ell$. See \cite{Vig96, Vig00a, Vig00b} for the modular setting and \cite{BH06} for the characteristic zero setting.

\subsection{Notation and the statement}
Consider the coefficients $\Lambda = \overline{\Q}_\ell$, $\overline{\F}_\ell$. In both cases we use the following standard notation. 
Let $\mathscr{A}^0_n(K, \Lambda)$ be the set of isomorphism classes of irreducible supercuspidal representations of $\mathbf{GL}_n(K)$ over $\Lambda$.
On the other side, let $\mathscr{G}^0_n(K, \Lambda)$ be the set of isomorphism classes of irreducible $n$-dimensional representations of the Weil group $\mathcal{W}_K$ over $\Lambda$.

For us, the local Langlands correspondence is a unique bijection
$$\mathscr{G}^0_n(K, \Lambda) \xrightarrow{\LL}  \mathscr{A}^0_n(K, \Lambda),$$
satisfying certain naturality properties. Its unique existence is known for $\mathbf{GL}_n(K)$ for both $\Lambda = \overline{\Q}_\ell$ and $\overline{\F}_\ell$. When $\Lambda = \overline{\Q}_\ell$ we call this simply the {\it local Langlands correspondence}; for the case $\Lambda = \overline{\F}_\ell$ we use the name {\it modular local Langlands correspondence}.

\subsection{The modular local Langlands correspondence and reduction mod \texorpdfstring{$\ell$}{l}}
The modular Local Langlands correspondence for $\mathbf{GL}_n(K)$ has been deduced by reduction modulo $\ell$ from the zero characteristic case by Vignéras \cite{Vig96, Vig00a, Vig00b}. A brief summary can be also found in \cite{Se15}. 

In addition to the notation of the previous section, we write
$$
\mathscr{A}^0_n(K, \overline{\Q}_\ell) \supseteq \mathscr{A}^0_n(K, \overline{\Q}_\ell)^{\integral} \supseteq \mathscr{A}^0_n(K, \overline{\Q}_\ell)^{\ell\rm{-scu}}
$$
for the consecutive subsets of $\mathscr{A}^0_n(K, \overline{\Q}_\ell)$ given by integral representations resp. integral representations with supercuspidal reduction modulo $\ell$.

Analogously let 
$$
\mathscr{G}^0_n(K, \overline{\Q}_\ell) \supseteq \mathscr{G}^0_n(K, \overline{\Q}_\ell)^{\integral} \supseteq \mathscr{G}^0_n(K, \overline{\Q}_\ell)^{\ell\rm{-irr}}
$$
be the consecutive subsets of $\mathscr{G}^0_n(K, \overline{\Q}_\ell) $ given by integral representations  resp. integral representations with irreducible reduction modulo $\ell$.
 
\begin{theorem}[Vignéras]\label{Vig}
\label{theorem: vigneras}
The following diagram commutes
\begin{center}
\begin{tikzcd}
\mathscr{A}^0_n(K, \overline{\Q}_\ell) \arrow[hookleftarrow]{r}   & \mathscr{A}^0_n(K, \overline{\Q}_\ell)^{\ell\rm{-scu}}  \arrow[two heads]{r}{r_\ell} & \mathscr{A}^0_n(K, \overline{\F}_\ell)  \\

\mathscr{G}^0_n(K, \overline{\Q}_\ell) \arrow[hookleftarrow]{r} \arrow{u}{\LL}[swap]{\cong} & \mathscr{G}^0_n(K, \overline{\Q}_\ell)^{\ell\rm{-irr}}  \arrow[two heads]{r}{r_\ell} \arrow{u}{\LL}[swap]{\cong}  & \mathscr{G}^0_n(K, \overline{\F}_\ell)  \arrow{u}{\LL}[swap]{\cong}
\end{tikzcd}
\end{center}
\end{theorem}
\begin{proof}
Up to reformulation, this is \cite[Theorem 4.6]{Se15}. For details, see \cite{Vig00a, Vig00b}. 
\end{proof}

To phrase this in words: the local Langlands correspondence preserves integral representations and identifies $\ell$-supercuspidal representations of $G$ with $\ell$-irreducible representations of $\mathcal{W}_K$. After the reduction $r_\ell$, it descends to a bijection realizing the modular local Langlands correspondence.

\subsection{Weil induction of characters}
\label{section: weil induction of characters}
\label{section: artin reciprocity identification}

We need the following partial parametrization of $\mathscr{G}^0_n(K, \Lambda)$. Let $L$ be the degree $n$ unramified extension of $K$ and $\mathscr{X}(L, \Lambda) \subseteq \Irr(T, \Lambda)$ the subset of characters in general position.
There is a map
\begin{equation}\label{definition: weil induction of characters}
\begin{gathered}
\sigma: \mathscr{X}(L, \Lambda)/\Gal(L/K) \to \mathscr{G}^0_n(K, \Lambda) \\   
\theta \mapsto \sigma(\theta) := \cInd_{\W_L}^{\W_K}\left(\W_L \to \W^{\ab}_L \xrightarrow{\rec_L} L^\times \xrightarrow{\mu \cdot \theta} \Lambda\right). 
\end{gathered}
\end{equation}
Here, the {\it rectifying character} $\mu$ is determined by sending a uniformizer $\varpi$ to $(-1)^{n-1}$ and being trivial on the compact part under the identification $\Irr(T, \Lambda) = \Lambda^\times \times \Irr(T_{\O}, \Lambda)$. The map $\rec_L$ is the Artin reciprocity map of class field theory. 

The following seems well-known, but we lack a reference in the modular case. For completeness, we sketch a proof.
\begin{proposition}\label{Wirreducible}
\label{proposition: weil induction}
The map \eqref{definition: weil induction of characters} is a well-defined injection.
\end{proposition}
\begin{proof}
Since $\mu$ is stabilized by the $\Gal(L/K)$-action, the adjustment $\theta \mapsto \mu \cdot \theta$ only permutes $\mathscr{X}(L, \Lambda)$ resp. $\mathscr{X}(L, \Lambda)/\Gal(L/K)$, and thus is irrelevant for the statement -- hence we ignore it.

Given a character $\theta \in \Irr(L^\times, \Lambda)$, denote $V$ the corresponding one-dimensional representation of $\W_L$ obtained via the reciprocity isomorphism of class field theory.

We have a short exact sequence $1 \to \W_L \to \W_K \to \Gal(L/K) \to 1$. In particular, $\W_L$ is normal in $\W_K$. For $g \in \Gal(L/K)$, we denote $^gV$ the twisted representation obtained from $V$ by precomposing with the conjugation $\W_L \to \W_L$, $w \mapsto \tilde{g}w\tilde{g}^{-1}$ by any lift $\tilde{g} \in \W_K$ of $g$; the isomorphism class of the resulting representation is independent of the choice of $\tilde{g}$. Therefore, $\Gal(L/K)$ acts on the character group of $\W_L$.

The Galois group $\Gal(L/K)$ also acts on the character group $\Irr(L^\times, \Lambda)$ via precomposition with its defining action on $L$. 
It is a standard result of local class field theory that the identification of the character groups via the reciprocity map is $\Gal(L/K)$-equivariant with respect to the above described actions, the relevant fact being that the reciprocity map comes from a morphism of modulations \cite[Definition 1.5.10]{Neu}.

Using the normality of $\W_L$ in $\W_K$, we obtain a Mackey decomposition
$$\Res^{\W_K}_{\W_L}\cInd^{\W_K}_{\W_L} V \cong \bigoplus_{g \in \Gal(L/K)} {^gV}.$$
Note this has dimension $n = |\Gal(L/K)|$. The terms $^gV$ on the right-hand side are one-dimensional, so in particular irreducible -- it follows that this $\W_L$-representation is semisimple. 

From now on, assume $\theta$ is in general position as in the statement. By the above discussion, the terms $^gV$ on the right hand-side are pairwise non-isomorphic. 

Consider any nonzero subrepresentation $U \subseteq \cInd^{\W_K}_{\W_L} V$. By the Krull--Schmidt theorem, $\Res^{\W_K}_{\W_L}U$ contains a $\W_L$-representation isomorphic to $^gV$ for some $g \in \Gal(L/K)$. For any $h \in \Gal(L/K)$, the translate of this by a lift of $hg^{-1}$ gives a representation of $\W_L$ isomorphic to $^hV$. Using Krull--Schmidt, this forces $\dim_{\Lambda} U \geq n$, showing $U = \cInd^{\W_K}_{\W_L} V$. Altogether, $\cInd^{\W_K}_{\W_L} V$ is irreducible.

We thus have a map $\mathscr{X}(L, \Lambda) \to \mathscr{G}^0_n(K, \Lambda)$ and it remains to see that it identifies precisely the Galois orbits. This is now immediate -- given two characters $\theta, \theta' \in \mathscr{X}(L, \Lambda)$ with corresponding $\W_L$-representations $V, V'$, Frobenius reciprocity and the Mackey formula above yield 
$$\Hom_{\W_K}\left( \cInd^{\W_K}_{\W_L} V, \cInd^{\W_K}_{\W_L} V' \right) \cong \Hom_{\W_L} \left( V,  \Res^{\W_K}_{\W_L}\cInd^{\W_K}_{\W_L} V' \right).$$ 
By the already proven irreducibility, the left-hand side is nonzero if and only if the inductions are isomorphic, whereas the right-hand side is nonzero if and only if $V$, $V'$ (hence the original characters $\theta$, $\theta'$) lie in the same $\Gal(L/K)$-orbit.
\end{proof}

We will also need a compatibility of the map \eqref{definition: weil induction of characters} for $\Lambda = \Ql$, $\Fl$ with the reduction $r_{\ell}$.

\begin{lemma}\label{weylind}
\label{lemma: weil induction and reduction}
The following well-defined diagram commutes
\begin{center}
\begin{tikzcd}
\mathscr{X}(L, \overline{\Q}_\ell)^{\sg, \integral} \arrow{d}{} \arrow{r}{r_\ell} &  \mathscr{X}(L, \overline{\F}_\ell)^{\sg} \arrow{d}{} \\
\mathscr{G}^0_n(K, \overline{\Q}_\ell)^{\integral} \arrow{r}{r_\ell} &  \mathscr{G}^0_n(K, \overline{\F}_\ell) 
\end{tikzcd}
\end{center}
\end{lemma}
\begin{proof}
Take $\theta \in \mathscr{X}(L, \overline{\Q}_\ell)^{\sg, \: \integral}$, and $\psi \in \mathscr{X}(L, \overline{\Q}_\ell)^{\sg}$ be its image under $r_\ell$. We have already seen that all of this makes sense -- the smooth character $\theta$ has an integral structure by design; $r_\ell$ preserves its strongly general position by Lemma \ref{Tposition}. The adjustment by $\mu$ does not affect integrality; the formation of the inflated representation to $\W_L$ commutes with $r_\ell$. Compact induction carries forward the integral structure and commutes with $r_\ell$ by Lemma \ref{lemma: cind and rl}, whose assumptions are satisfied by Discussion \ref{reductionmodulolitems} (ii) -- indeed, $|\W_K / \W_L| = n$ is finite, so all representations in question are finite-dimensional. The diagram thus commutes. 

Finally, we should point out that the vertical maps land in irreducible representations of $\W_K$ by Proposition \ref{Wirreducible}.
\end{proof}
\section{Group actions on étale cohomology}
\label{section: equivariant etale cohomology}

For an algebraic variety $X$, we have its compactly supported étale cohomology groups $H^\bullet_c(X, \Lambda)$; these are the cohomology groups of the complex $\RG_c(X, \Lambda) \in \D^b(\Lambda)$.
When $X$ carries an action of a finite group $G$, we get an induced action on its étale cohomology: $\RG_c(X, \Lambda) \in \D^b(G, \Lambda)$.
 
Rickard \cite{Ric94} noticed that this complex has a canonical representative $\GG_c(X, \Lambda)$ in the homotopy category $\K^b(G, \Lambda)$ of a specific shape. Many constructions on $\RG_c(X, \Lambda)$ -- which are necessarily derived in nature -- can be performed directly on this representative. This makes $\GG_c(X, \Lambda)$ an important tool for understanding the behaviour of representations coming from cohomology under the change of coefficients.
We briefly review the properties of Rickard's complexes \cite{Ric94}. The relevance for classical modular Deligne--Lusztig theory is considered in \cite{BR02, BDR16, Dud17}.

\subsection{Conventions}\label{conventions: sheaves on schemes}

Let $X$ be a separated scheme of finite type over a separably closed field $k$. Let $\Lambda$ be as in Setup \ref{setting: coefficient rings}. All cohomology groups are étale cohomology groups with compact support or their $\ell$-adic version \cite{SGA4, Mil}. We denote them $H^\bullet_c(X, \Lambda)$.

Given a coefficient ring $\A$ as in Setup \ref{setting: noncommutative coefficient rings}-(i), we denote $\Sh(X, \A)$ the category of constructible étale sheaves of $\A$-modules on $X$. We denote its bounded homotopy category by $\K^b(X, \A)$ and its bounded derived category by $\D^b(X, \A)$. By $\RG_c(X, -): \D^b(X, \A) \to \D^b(\A)$, we denote the derived functor of compactly supported global sections. 
For a coefficient ring $A$ as in Setup \ref{setting: coefficient rings}-(ii), we use $\Sh(X, A)$ for the corresponding category of $\ell$-adic sheaves.

If $\mathcal{M} \subseteq \Mo(\A)$ is an idempotent complete abelian subcategory, we denote $\Sh(X, \mathcal{M}) \subseteq \Sh(X, \A)$ the full subcategory of sheaves with stalks in $\mathcal{M}$. We further denote $\K^b(X, \mathcal{M})$ and $\D^b(X, \mathcal{M})$ the subcategories of objects which can be represented by a complex with entries in $\Sh(X, \mathcal{M})$.

We write $\iota: \K^b(-) \to \D^b(-)$ for the canonical localization functor.

\subsection{Rickard's complex} \label{subsubsection: rickards complex} \label{section: equivariant cohomology via rickards complex}
In order to effectively manipulate $\RG_c(X, \Lambda)$, we need to work with explicit representatives; in order to perform certain operation on $\RG_c(X, \Lambda)$, such a representative has to be sufficiently acyclic. Rickard \cite{Ric94} constructed a useful canonical representative $\KG_c(X, \Lambda)$ for $\RG_c(X, \Lambda)$ on the level of the homotopy category. See also \cite{Dud17}.
 
Let $X$ be a separated scheme of finite type over a field $k$ and $\A$ a coefficient ring as in Setup \ref{setting: noncommutative coefficient rings}. 
Rickard's construction gives a functor
\begin{equation*}
 \KG_c(X, -): \Sh(X, \A) \to \K^b(\A)  
\end{equation*}
with the following properties.
\begin{properties}[Rickard's complex]\label{properties: rickards complex}
\ 
\begin{enumerate}
\item \label{property: representing cohomology} Rickard's functor $\KG_c(X, -)$ is additive and gives a functorial representative for $\RG_c(X, -)$ on the level of $\K^b(A)$. In other words, $\RG_c(X, -) \cong \iota(\KG_c(X, -))$.

\item For any $\mathscr{F} \in \Sh(X, \A)$, the complex $\KG_c(X, \mathscr{F})$ has finite type terms and is concentrated in degrees $0, \dots, 2\dim X$. More precisely, it is concentrated in the interval of degrees where $H^i_c(X, \mathscr{F}) \neq 0$.

\item \label{property: stalks and additive subcategories} If all stalks of $\mathscr{F} \in \Sh(X, \A)$ lie in $\add(M)$ for some $M \in \mo(A)$, then also  $\KG_c(X, \mathscr{F})$ lands in $\K^b(\add(M))$. Put in different words, if $\mathcal{M}$ is an additive idempotent complete subcategory of $\mo(\A)$, then $\KG_c(X, -)$ restrict to a functor $\Sh(X, \mathcal{M}) \to \K^b(\mathcal{M})$.

\item \label{property: functoriality in the first variable} Rickard's complex $\KG_c(-, -)$ is functorial in the first variable in the following sense: a finite morphism of pairs $(X, \mathscr{F}) \to (Y, \mathscr{G})$ induces a map $\KG_c(Y, \mathscr{G}) \to \KG_c(X, \mathscr{F})$.

\item \label{property: compatibility with additive functors} If $F: \mo(\A) \to \mo(\B)$ is an additive functor, we have a canonical isomorphism
$$F(\KG_c(X, \mathscr{F})) = \KG_c(X, F^{\sh}(\mathscr{F}))$$ 
inside $\K^b(\B)$. Here, $F^{\sh}(\mathscr{F})$ stands for the sheafification of the presheaf given by applying $F$ sectionwise.
\end{enumerate}
\end{properties}
\begin{proof}
See \cite[Theorem 2.7 and Lemma 2.8]{Ric94} for coefficients from Setup \ref{setting: noncommutative coefficient rings}-(i). The extension to coefficients from Setup \ref{setting: noncommutative coefficient rings}-(ii) is done as in \cite[Theorem 3.5]{Ric94} and further base change.
\end{proof}

\begin{remark}
Property \eqref{property: compatibility with additive functors} is quite strong.
For instance -- together with property \eqref{property: representing cohomology} -- it reveals that $\RG_c(X, F^{\sh}(\mathscr{F}))$ may be computed by naively applying $F$ to the representing complex $\KG_c(X, \mathscr{F})$.
Also, \eqref{property: compatibility with additive functors} shows that the construction of $\KG_c(X, -)$ is compatible with forgetful functors between module categories. Consequently, we can dismiss these forgetful maps from the notation. 
\end{remark}

\begin{remark}
Without any difficulty, one can extend $\KG_c(X, -)$ to a functor $\K^b(X, \A) \to \K^b(\A)$.
\end{remark}

\subsection{Finite group actions}
\label{subsection: finite group actions}
Let $X$ be a separated scheme of finite type over a separably closed field $k$.
Consider an action of a finite group $G$ on $X$. 

\begin{notation}[Stack quotients]
We denote $\langle X \slash G \rangle$ the stack quotient of $X$ by $G$. 
The conventions from \S\ref{conventions: sheaves on schemes} extend to algebraic stacks. 
\end{notation}

\begin{notation}[Coarse quotients]
We denote $X \slash G$ the quotient of $X$ by $G$ in the sense of algebraic spaces. This gives a course moduli space for $\langle X \slash G \rangle$ by \cite{KM97, Ryd13}.

When $X$ is a quasi-projective scheme over $k$, so is $X \slash G$. It can be constructed as follows: find a $G$-invariant affine covering of $X$ (which exists by quasi-projectivity), take invariants $(-)^G$ on functions and then glue back \cite[Appendix A]{Mus11}.
This is the approach taken in \cite{Ric94, DL76}; it is sufficient for this paper.
\end{notation}

A {\it $G$-equivariant sheaf} is an object $\eF \in \Sh(X, \Lambda)$ equipped with isomorphisms $\alpha_g: \mathscr{F} \to (g^{-1})^*\mathscr{F}$ for each $g \in G$ satisfying the cocycle condition. More conceptually, this amounts to an object $\eF \in \Sh(\langle X / G \rangle, \Lambda)$.
In particular, $\Sh(\langle \pt / G \rangle, \Lambda)$ is the category of finite type $G$-representations; similarly for $\D^b(\langle \pt / G \rangle, \Lambda)$.

The cohomology of an equivariant sheaf carries a natural $G$-action: denoting $f: \langle X / G \rangle \to \langle \pt / G \rangle$ the canonical map, base change allows to identify
\begin{equation*}
 \RG_c(X, \eF) = Rf_{!}\eF \in \D^b(G, \Lambda).    
\end{equation*}
Taking the constant sheaf $\Lambda$ with the trivial $G$-equivariant structure, we get
$\RG_c(X, \Lambda) \in \D^b(G, \Lambda)$.

Here is an alternative way of computing this complex via the coarse quotient. Let $\pi: X \to X / G$ be the quotient map. Given a $G$-equivariant sheaf $\eF$ in $\Sh(X, \Lambda)$, the pushforward $\pi_* \eF$ naturally lives in $\Sh(X/G, \Lambda[G])$.
Since $\pi$ is finite, $R\pi_* = \pi_* = \pi_{!} = R\pi_{!}$ has no higher cohomology. Given any further map $h: X/G \to S$ to some scheme $S$, we deduce
\begin{equation*}
 R(h \circ \pi)_! = Rh_! R\pi_!   
\end{equation*}
In particular,
\begin{equation}\label{equation: pushforward to quotient}
\RG_c(X, \eF) = \RG_c(X/G, \pi_* \eF) \in \D^b(G, \Lambda).    
\end{equation}

\begin{notation}
Given a point $x \in X(k)$, we write $C_G(x)$ for its stabilizer under the $G$-action. Denoting $y= \pi(x)$ its image under the quotient map $\pi: X \to X / G$, the stalk of the pushforward of $\Lambda$ identifies as $(\pi_*\Lambda)_y \cong \Lambda[G/C_G(x)]$.
\end{notation}

\begin{definition}[Rickard's complex of a group action]
For a finite group $G$ acting on $X$ and a $G$-equivariant sheaf $\eF$, denote
\begin{equation*}
    \GG_c(X, \eF) := \KG_c(X \slash G, \pi_* \eF) \in \K^b(G, \Lambda)
\end{equation*}
By \S \ref{subsubsection: rickards complex}, this is a complex of finite type $\Lambda[G]$-modules representing $\RG_c(X, \eF) \in \D^b(G, \Lambda)$.
\end{definition}

\begin{lemma}\label{corollary: rickards complex for finite group actions}
For a finite group $G$ acting on $X$, the complex
$\GG_c(X, \Lambda) \in \K^b(G, \Lambda)$
has terms in $\mathcal{M} = \add \{\Lambda[G/C_G(x)] \mid x \in X(k)\}$. 
\end{lemma}
\begin{proof}
This follows from the Property \ref{properties: rickards complex}-\eqref{property: stalks and additive subcategories}, because the stalks of the relevant sheaf $\pi_* \Lambda$ on $X/G$ are precisely of the form $\Lambda[G/C_G(x)]$.
See \cite[Theorems 3.2 and 3.5]{Ric94}.
\end{proof}

\begin{lemma}\label{lemma: rickard complexes and coefficients}
Let $\Lambda$ be as in Setup \ref{setting: coefficient rings}. Then on the level of bounded homotopy categories
$$\GG_c(X, \Lambda) \cong \GG_c(X, \Z_\ell) \otimes_{\Z_\ell} \Lambda .$$
\end{lemma}
\begin{proof}
This is an instance of Property \ref{properties: rickards complex}-\eqref{property: compatibility with additive functors} for $F = (-\otimes_{\Z_\ell} \Lambda)$.
\end{proof}

\begin{corollary}[Euler characteristic and reduction]\label{remark: euler characteristic and reduction}
The map $r_\ell: \G_0(G, \Ql) \to \G_0(G, \Fl)$ sends
\begin{equation}\label{equation: euler characteristic and reduction}
r_\ell: [H^\bullet_{c}(X, \Ql)] \mapsto [H^\bullet_{c}(X, \Fl)].
\end{equation}
\end{corollary}
\begin{proof}
Indeed, this can be rewritten as $r_\ell: [\GG_c(X, \Ql)] \mapsto [\GG_c(X, \Fl)]$. By Lemma \ref{corollary: rickards complex for finite group actions}, $\GG_c(X, \Zl)$ is a complex of torsionfree $\Zl$-modules; by Lemma \ref{lemma: rickard complexes and coefficients} it thus gives an integral structure of $\GG_c(X, \Ql)$ which reduces to $\GG_c(X, \Fl)$ modulo $\ell$ and the statement follows.
\end{proof}

Corollary \ref{remark: euler characteristic and reduction} can be also checked directly without Rickard's complexes. In particular if $G$ is trivial, it amounts to the standard equality of Betti numbers of $X$ with $\Ql$ and $\Fl$ coefficients \cite[p. 166]{Mil}.
We will employ similar reasoning in Corollary \ref{corollary: equivariant euler characterictic and reduction} for derived isotypic parts.

\subsection{Derived isotypic parts}\label{subsection: derived isotypic parts}
Let $Y$ be a separated scheme of finite type over a separably closed field $k$. Let $G$ and $H$ be finite groups. Suppose $Y$ carries commuting actions of $G$ and $H$. In other words, $G\times H$ act on $Y$ and we get the complexes
\begin{equation*}
\GG_c(Y, \Lambda) \in \K^b(G \times H, \Lambda)
\qquad\mathrm{and}\qquad 
\RG_c(Y, \Lambda) \in \D^b(G \times H, \Lambda).
\end{equation*}
We now want to make sense of the isotypic parts for the $H$-action on them.

\begin{notation}\label{definition: rickards isotypic parts}
For $\theta \in \Rep(H, \Lambda)$, we denote
$$(-)_\theta := (- \otimes_{\Lambda[H]} \theta ): \Rep(G \times H, \Lambda) \to \Rep(G, \Lambda)$$ 
the usual tensor product with $\theta$. We use the same notation for the induced functor on homotopy categories
$$(-)_\theta := (- \otimes_{\Lambda[H]} \theta ): \K^b(G \times H, \Lambda) \to \K^b(G, \Lambda).$$
We call this $(-)_\theta$ the {\it isotypic part} of $\theta$.
\end{notation}
\begin{notation}\label{definition: derived isotypic parts}
For $\theta \in \Rep(H, \Lambda)$, we write
$$(-)_{\eL \theta} := (- \otimes^{\eL}_{\Lambda[H]} \theta ): \D^b(G \times H, \Lambda) \to \D^b(G, \Lambda)$$
for the derived tensor product with $\theta$. We call $(-)_{\eL \theta}$ the {\it derived isotypic part} of $\theta$.
\end{notation}

It turns out that when $H$ acts freely, the derived isotypic parts of $\RG_c(Y, \Lambda)$ may be computed as the non-derived isotypic parts of $\GG_c(Y, \Lambda)$. 
\begin{lemma}\label{lemma: abstract derived and nonderived isotypic parts}
Suppose $Y$ carries commuting actions of finite groups $G$ and $H$. Assume that $H$ acts freely on $Y$.
Let $\theta \in \Rep(G, \Lambda)^{\ft}$. Then on the level of $\D^b(G, \Lambda)$, we have
$$\RG_c(Y, \Lambda)_{\eL \theta} \cong \iota(\GG_c(Y, \Lambda)_\theta).$$
\end{lemma}

\begin{proof}
Denote $\rho:Y \to Y/(G \times H)$ the quotient map. We compute:
\begin{align*}
\RG_c(Y, \Lambda)_{\eL \theta} 
&= \RG_c(Y/(G \times H), \rho_* \Lambda) \otimes^{\eL}_{\Lambda [H]} \theta \\
& = \RG_c(Y/(G \times H), \rho_*\Lambda\otimes^{\eL}_{\Lambda [H]} \theta) \\
& = \RG_c(Y/(G \times H), \rho_*\Lambda\otimes_{\Lambda [H]} \theta) \\
& = \iota(\KG_c(Y/(G \times H), \rho_*\Lambda\otimes_{\Lambda [H]} \theta)) \\
& = \iota(\KG_c(Y/(G \times H), \rho_*\Lambda)\otimes_{\Lambda [H]} \theta) \\
& = \iota(\GG_c(Y, \Lambda)_\theta)
\end{align*}
Indeed, the first equality is just the definition of $(-)_{\eL \theta}$ together with \eqref{equation: pushforward to quotient}. The second line is a general property of étale cohomology \cite[XVII, 5.2.9]{SGA4}. The third equality uses the freeness of the $H$-action on $Y$: each stalk of $\rho_*\Lambda$ is the permutation module on the corresponding $(G\times H)$-orbit inside $Y$, so in particular free as $\Lambda[H]$-module -- we thus do not need to derive the tensor product. The fourth line holds by Property \ref{properties: rickards complex}-\eqref{property: representing cohomology}, whereas the fifth one is Property \ref{properties: rickards complex}-\eqref{property: compatibility with additive functors}. The last line is again just the definition of $(-)_\theta$.
\end{proof}

\begin{remark}\label{remark: abstract description via local systems}
Denoting $\pi: Y \to Y/H$ be the obvious quotient map, we can alternatively think of $\RG(Y, \Lambda)_{\eL \theta}$ as the cohomology $\RG_c(Y/H, \Lambda_{\theta})$ of the $G$-equivariant local system 
$$\Lambda_{\theta} := \pi_* \Lambda \otimes_{\Lambda[T]} \theta$$
on $Y / H$. 
\end{remark}
\begin{proof}[Proof of Remark \ref{remark: abstract description via local systems}]
This is clear say from the third or fourth line of the above computation: to conclude, one only needs to factor $\rho$ as
$$Y \xrightarrow{\pi} Y / H \xrightarrow{\rho'} Y/(G \times H) $$
and identify the coefficient sheaves
$$\rho_*\Lambda\otimes_{\Lambda [H]} \theta = \rho'_* \pi_* \Lambda\otimes_{\Lambda [H]} \theta = \rho'_*( \pi_* \Lambda\otimes_{\Lambda [H]} \theta)$$
by commuting $(-)_{\theta}$ with the complementary pushforward on stalks.
The isotypic part from Lemma \ref{lemma: abstract derived and nonderived isotypic parts} thus indeed identifies with $\RG_c(Y/H, \Lambda_{\theta}) = \GG_c(Y/H, \Lambda_{\theta}).$ 
\end{proof}
\section{\texorpdfstring{$p$}{p}-adic Deligne--Lusztig spaces and their cohomology}
\label{section: p-adic deligne--lusztig theory}
In \S \ref{section: overview of p-adic Deligne--Lusztig spaces} we provide a brief account of the $p$-adic Deligne--Lusztig spaces of \cite{Iva20b}. 
In \S \ref{section: the concrete case of GL_n} we specialize to the case of our interest for $GL_n$ following \cite{Iva19, Iva18}. In \S \ref{section: isotypic parts at finite level} we discuss étale cohomology of the truncated spaces $\dot{X}_h$ with coefficients in any $\Lambda$ from Setup \ref{setting: coefficient rings} and its derived isotypic parts. We use Rickard's complexes to control their behaviour with respect to reduction $r_{\ell}$. 
In \S \ref{section: equivariant cohomology of deligne--lusztig spaces} we consider the isotypic parts of the cohomology of whole space $\dot{X}$. 
In \S \ref{section: representation stability of the truncated deligne--lusztig spaces} we show that this definition is independent of $h$ up to an explicit even cohomological shift. 
In \S \ref{section: finteness properties of the cohomology} we prove that it gives a well-defined class in the Grothendieck group of finite length representations for both $\Lambda = \Ql$ and $\Fl$. During the process, we prove that these classes are compatible under the reduction $r_{\ell}$. 
In \S \ref{section: realization of the local langlands correspondence} we review the partial realization of the local Langlands correspondence in characteristic zero via $p$-adic Deligne--Lusztig spaces of \cite{Iva19}. 
In \S \ref{section: main result} we deduce our main result -- Theorem \ref{theorem: upshot}. In the complementary \S \ref{section: a version with multiplicities} we present a more naive version of this theorem, resulting in an extra multiplicity for small $\ell$.

\subsection{Overview of \texorpdfstring{$p$}{p}-adic Deligne--Lusztig spaces}\label{section: overview of p-adic Deligne--Lusztig spaces}
Let $K$ be a local non-archimedean field as in \S \ref{notation: non-archimedean fields}.
For a scheme $X$ over $K$, the {\it loop space} $LX$ is the presheaf on the opposite category of perfect $k$-algebras given by
$LX: R \mapsto X(\mathbb{W}(R)[\varpi^{-1}]).$
When $X$ actually lives over $\O_K$, the {\it positive loop space} $L^+X$ resp. the {\it truncated loop spaces} $L^+_hX$ for $h \in \N$ defined as
$L^+X: R \mapsto X(\mathbb{W}(R))$ resp. $L^+_hX: R \mapsto X(\mathbb{W}_h(R))$. See \cite{Zhu16, Iva20b}.

Let $\mathbf{G}$ be an unramified reductive group over $K$. Fix a maximal $K$-rational maximally split torus $\mathbf{T}$, which splits after an unramified extension,
and a $K$-rational Borel $\mathbf{B} = \mathbf{U}\mathbf{T}$ with unipotent radical $\mathbf{U}$. Let $\breve{K}$ be the maximal unramified extension of $K$ and $\sigma \in \Aut(\breve{K} / K)$ the Frobenius lift. Then $\mathbf{G}$ and $\mathbf{T}$ split over $\breve{K}$ and we denote $N$ resp. $W$ the normalizer resp. the Weyl group of $\mathbf{T}_{\breve{K}}$ in $\mathbf{G}_{{\breve{K}}}$. 
The Bruhat decomposition yields a stratification of $\mathbf{G}/\mathbf{B} \times \mathbf{G}/\mathbf{B}$ into locally closed subspaces $\O(w)$, $w \in W$. The bigger space $\mathbf{G}/\mathbf{U} \times \mathbf{G}/\mathbf{U}$ can be decomposed into $\O(\dot{w})$, $\dot{w} \in N_G(T)$.

Fix $b \in \mathbf{G}(\breve{K})$, and $\dot{w} \in N$ resp. its image $w \in W$. The $p$-adic Deligne--Lusztig spaces are defined in \cite[Definition 8.3]{Iva20b} as the following fiber products in presheaves on the opposite category of perfect algebras:
\begin{equation}\label{equation: definition of deligne-lusztig spaces}
\begin{tikzcd}
\dot{X}_{\dot{w}}(b) \arrow{d}{} \arrow{r}{} &  L(\O(\dot{w})) \arrow{d} \\
L(\mathbf{G}/\mathbf{U}) \arrow{r}{\id \times b\sigma } &  L(\mathbf{G}/\mathbf{U}) \times L(\mathbf{G}/\mathbf{U}) 
\end{tikzcd}
\qquad
\begin{tikzcd}
X_w(b) \arrow{d}{} \arrow{r}{} &  L(\O(w)) \arrow{d} \\
L(\mathbf{G}/\mathbf{B}) \arrow{r}{\id \times b\sigma } &  L(\mathbf{G}/\mathbf{B})\times L(\mathbf{G}/\mathbf{B})  
\end{tikzcd}
\end{equation}
They depend up to isomorphism only on the $\sigma$-conjugacy class $[b] := \{ g^{-1}b\sigma(g) \mid g \in \mathbf{G}(\breve{K}) \}$ of $b$. 
There is a natural map
\begin{equation}\label{equation: T_w torsor}
   \dot{X}_{\dot{w}}(b) \to X_{w}(b).
\end{equation}
The spaces \eqref{equation: T_w torsor} carry actions of the profinite groups
$$G_b := \mathbf{G}_b(K) = \{ g \in \mathbf{G}(\breve{K}) \mid g^{-1} b \sigma(g) = b \}, $$
$$T_w := \mathbf{T}_w(K) = \{ t \in \mathbf{T}(\breve{K}) \mid t^{-1} \dot{w} \sigma(t) = \dot{w} \}.$$
More precisely, both $\dot{X}_{\dot{w}}(b)$ and $X_{w}(b)$ carry an action of $G_b$ such that \eqref{equation: T_w torsor} is $G_b$-equivariant \cite[\S 8.2.1]{Iva20b}.
Furthermore, $\dot{X}_{\dot{w}}(b)$ carries a $T_w$-action, commuting with the above $G_b$-action \cite[\S 8.2.2]{Iva20b}. The map \eqref{equation: T_w torsor} is a $T_w$-torsor \cite[Proposition 11.9]{Iva20b}.

Given a reductive $\O_K$-model of $\mathbf{G}$, we can replace the loop space $L(-)$ in \eqref{equation: definition of deligne-lusztig spaces} by the integral and truncated loop spaces $L^+(-)$ and $L^+_h(-)$, $h \in \N$. This yields the presheaves 
\begin{equation*}
\dot{X}_{\dot{w}}(b)^+ \qquad \text{and} \qquad \dot{X}_{\dot{w}}(b)^+_h 
\qquad \text{resp.} \qquad
X_{w}(b)^+ \qquad \text{and} \qquad X_{w}(b)^+_h.
\end{equation*}
When $h = 1$, these truncated spaces are representable by the (perfections of) usual Deligne--Lusztig varieties. In this sense, the above definition generalizes the standard setting.

In many cases, $p$-adic Deligne--Lusztig spaces are ind-representable by perfect pfp schemes \cite[Theorem 9.1, Corollary 9.2]{Iva20b}. Moreover, pfp perfect schemes have a well-behaved theory of étale sheaves equipped with six operations \cite[Appendix A, \S A.3]{Zhu16}; it agrees with the theory of étale sheaves on finite type models. In particular, compactly supported cohomology of such spaces is well-defined.

\subsection{The case of our interest}\label{section: the concrete case of GL_n}
From now on, we specialize to the case of split $\mathbf{G} = \mathbf{GL}_n$ over $K$ and its hyperspecial parahoric model $\mathbf{G}_{\O} = \mathbf{GL}_n$ over $\O_K$. We put $b=1$ and consider a Coxeter element $w \in W$, we take $\dot{w} \in N$ to be a lift of $w$ with the same Kottwitz invariant as $b$. This is the most elementary case among those studied in \cite{Iva19, Iva18, Iva20b}.
For brevity, we denote
$$\dot{X} := \dot{X}_{\dot{w}}(b) \qquad \text{resp.}\qquad \dot{X}_{\O} := \dot{X}_{\dot{w}}(b)^+ \qquad \text{resp.} \qquad \dot{X}_h := \dot{X}_{\dot{w}}(b)^+_h,$$
$$X := X_{w}(b) \qquad \text{resp.}\qquad X_{\O} := X_{w}(b)^+ \qquad \text{resp.} \qquad X_h := X_{w}(b)^+_h.$$
These spaces satisfy \cite[\S 2.6]{Iva19}
\begin{equation} \label{equation: decomposition of Deligne--Lusztig sheaves}
\dot{X} \cong \coprod_{G/G_{\O}} \dot{X}_{\O} \qquad \text{and} \qquad \dot{X}_{\O} \cong \lim_{h} \dot{X}_h.
\end{equation}
In accordance with \S \ref{section: overview of p-adic Deligne--Lusztig spaces} they bear commuting actions of the following groups:
\begin{itemize}
\item[•] $\dot{X}$ has commuting actions of the locally profinite groups $G = \mathbf{GL}_n(K)$ and $T = \mathbf{T}_w(K) = L^\times$;
\item[•] $\dot{X}_{\O}$ has commuting action of the profinite subgroups $G_{\O} = \mathbf{GL}_n(\O_K)$ and $T_{\O} = \O_L^\times$;
\item[•] $\dot{X}_h$, $h \in \N$ have commuting actions of the finite quotients 
$$G_h = \mathbf{GL}_n(\O_K)/ (1 + M_{n \times n}(\mathfrak{p}_K^h)) \qquad \text{and} \qquad T_h = \O_L^\times/T^h = \O_L^\times / (1 + \mathfrak{p}^h_L).$$
\end{itemize}
These actions are compatible with the isomorphisms \eqref{equation: decomposition of Deligne--Lusztig sheaves} in the obvious way and the morphism \eqref{equation: T_w torsor} from $\dot{X}$ to $X$ is equivariant.

\begin{proposition}[{\cite[Proposition 7.4]{Iva18}}] \label{proposition: representability by smooth affine schemes}
The truncated spaces $\dot{X}_h$ for $h \in \N$ are representable by perfections of smooth affine $\overline{\F}_q$-schemes of finite type. 
\end{proposition}

\begin{remark}\label{fintype}
In particular, there is a good notion of compactly supported étale cohomology for these finite level $\dot{X}_h$ and the discussion of \S \ref{section: equivariant etale cohomology} applies. The compactly supported cohomology of $\dot{X}_h$ is concentrated only in the range $[\dim \dot{X}_h, 2 \dim \dot{X}_h]$ by affineness. The Poincaré duality between $H^{\bullet}$ and $H_c^{\bullet}$ holds by smoothness.
\end{remark}
The discussion in \S \ref{section: overview of p-adic Deligne--Lusztig spaces} implies the following.
\begin{proposition}\label{proposition: T_h torsor}
The $T_h$-action on $\dot{X}_h$ is free and the quotient map $\pi: \dot{X}_h \to X_h$ is a $T_h$-torsor.
\end{proposition}
We now remind the reader about the notation $T^{h-1}_h = \ker(T_h \to T_{h-1})$ from \S \ref{Tsec -- prof}.
The maps $\dot{X}_h \to \dot{X}_{h-1}$ factor as follows.
\begin{proposition}[{\cite[Proposition 7.7]{Iva18}}]\label{fib}
\label{proposition: affine fibration}
The morphism $\dot{X}_h \to \dot{X}_{h-1}$ factors through a perfectly smooth map $\dot{X}_h / T^{h-1}_h \to \dot{X}_{h-1}$ whose fibers are isomorphic to the perfection of $\mathbb{A}^{n-1}$.
\end{proposition}

\subsection{Isotypic parts at finite level} \label{section: isotypic parts at finite level}
Let us now discuss the equivariant cohomology of the truncated spaces $\dot{X}_h$, $h \in \N$ and their derived isotypic parts.
These finite level phenomena play a crucial for the discussion of the cohomology of the whole space $\dot{X}$ in \S \ref{section: equivariant cohomology of deligne--lusztig spaces}.

The spaces $\dot{X}_h$ and $X_h$ are equipped with commuting actions of $G_h$ and $T_h$. Their compactly supported cohomology is well-defined and inherits an action of $G_h \times T_h$. Appealing to \S \ref{subsection: finite group actions} for some finite type models, we get the complexes
\begin{equation*}
\GG_c(\dot{X}_h, \Lambda) \in \K^b(G_h \times T_h, \Lambda)
\qquad\mathrm{and}\qquad 
\RG_c(\dot{X}_h, \Lambda) \in \D^b(G_h \times T_h, \Lambda).
\end{equation*}
We now want to consider the isotypic parts for the $T_h$-action on them. In the semisimple case, this takes the form of a direct sum decomposition indexed by characters. In the non-semisimple case, one needs to be careful in choosing the right definition -- see Remark \ref{remark: definitions of isotypic parts}. To not distract the reader, we simply present the correct notions: for any $\theta \in \Hom_{\Grp}(T_h, \Lambda^\times)$, the discussion from \S \ref{subsection: derived isotypic parts} gives the complexes
\begin{equation*}
\GG_c(\dot{X}_h, \Lambda)_{\theta} \in \K^b(G_h, \Lambda)
\qquad\mathrm{and}\qquad 
\RG_c(\dot{X}_h, \Lambda)_{\eL \theta} \in \D^b(G_h, \Lambda).
\end{equation*}

In our situation, the derived isotypic parts of $\RG_c(\dot{X}_h, \Lambda)$ may be computed as the non-derived isotypic parts of $\GG_c(\dot{X}_h, \Lambda)$. This plays an important role in our comparison.
\begin{lemma}\label{lemma: derived and nonderived isotypic parts}
On the level of $\D^b(G_h, \Lambda)$, we have 
$$\RG_c(\dot{X}_h, \Lambda)_{\eL \theta} \cong \iota(\GG_c(\dot{X}_h, \Lambda)_\theta).$$
\end{lemma}
\begin{proof}
Since the action of $T_h$ on $\dot{X}_h$ is free by Proposition \ref{proposition: T_h torsor}, the claim follows from Lemma \ref{lemma: abstract derived and nonderived isotypic parts} applied to the action of $G_h \times T_h$ on $\dot{X}_h$.
\end{proof}

\begin{remark}
Note that this argument applies to other truncated Deligne--Lusztig spaces of \S \ref{section: overview of p-adic Deligne--Lusztig spaces}.  
\end{remark}

\begin{remark}[Description via local systems]\label{remark: description via local systems}
Denoting $\pi: \dot{X}_h \to X_h$ the obvious quotient map, we can alternatively think of $\RG(\dot{X}_h, \Lambda)_{\eL \theta}$ as the cohomology $\RG_c(X_h, \Lambda_{\theta})$ of the $G_h$-equivariant local system 
$$\Lambda_{\theta} := \pi_* \Lambda \otimes_{\Lambda[T_h]} \theta$$
on $X_h = \dot{X}_h / T_h$. Indeed, this follows from Remark \ref{remark: abstract description via local systems}.
\end{remark}

Lemma \ref{lemma: derived and nonderived isotypic parts} allows to compute $\RG(\dot{X}_h, \Lambda)_{\eL \theta}$ by applying the additive functor $(-)_{\theta}$ to the concrete representative $\GG(\dot{X}_h, \Lambda)$ in the homotopy category. This procedure is compatible with base changes of the coefficient ring $\Lambda$, so in particular with reduction modulo $\ell$. 
\begin{notation}
We abuse notation and write
$[H^\bullet_c(\dot{X}_h, \Lambda)_{\eL \theta}] := [\RG_c(\dot{X}_h, \Lambda)_{\eL \theta}] \in \G_0(G_h, \Lambda)$.
\end{notation}

\begin{corollary}\label{corollary: equivariant euler characterictic and reduction}
The reduction map $r_\ell: \G_0(G_h, \Ql) \to \G_0(G_h, \Fl)$ sends
$$[H^\bullet_c(\dot{X}_h, \Ql)_{\theta}] \mapsto [H^\bullet_c(\dot{X}_h, \Fl)_{\eL r_\ell(\theta )}].$$
\end{corollary}
\begin{proof}
Let $\theta \in \Irr(T_h, \Ql)$ regarded as an element in $\Rep(T_h, \Zl)^{\ft}$; then $(\theta \otimes_{\Zl} \Fl) = r_\ell(\theta )$ by design.
By Lemma \ref{lemma: rickard complexes and coefficients}, the complex $\GG_c(\dot{X}_h, \Zl)$ specializes to analogous complexes with $\Ql$ and $\Fl$ coefficients under the obvious base changes. 
Tensoring with $\theta$ and using the associativity of the tensor product, we obtain the following identification.
\begin{equation*}
\begin{tikzcd}
\K^b(G_h, \Ql) & \K^b(G_h, \Zl) \arrow{l}[swap]{(-\otimes_{\Zl} \Ql)} \arrow{r}{(-\otimes_{\Zl} \Fl) } & \K^b(G_h, \Fl) \\
\GG_c(\dot{X}_h, \Ql)_{\theta} & \GG_c(\dot{X}_h, \Zl)_{\theta} \arrow[mapsto]{l} \arrow[mapsto]{r} & \GG_c(\dot{X}_h, \Fl)_{r_\ell ( \theta )}
\end{tikzcd}
\end{equation*}
Since $T_h$ acts freely on $\dot{X}_h$, the complex $\GG_c(\dot{X}_h, \Zl)$ actually consists of projective $\Zl[T_h]$-modules by Lemma \ref{corollary: rickards complex for finite group actions}. The map $\GG_c(\dot{X}_h, \Zl)_{\theta} \to \GG_c(\dot{X}_h, \Ql)_{\theta}$ is thus injective, so $\GG_c(\dot{X}_h, \Zl)_{\theta}$ defines an integral structure of $\GG_c(\dot{X}_h, \Ql)_{\theta}$ which reduces to $\GG_c(\dot{X}_h, \Fl)_{r_{\ell} (\theta)}$ modulo $\ell$. 
\end{proof}

\begin{remark}\label{remark: definitions of isotypic parts}
When the category $\rep(T_h, \Lambda)$ is semisimple, the complexes $\GG(\dot{X}_h, \Lambda)$ decompose into a direct sum of isotypic components for the $T_h$-action and everything can be computed on the graded vector space $H^{\bullet}_c(\dot{X}_h, \Lambda)$. 
When $\rep(T_h, \Lambda)$ is not semisimple, this is not the case. Although one can still consider naive definitions of isotypic components on the graded vector space $H^{\bullet}_c(\dot{X}_h, \Lambda)$, they do not behave well with respect to geometric arguments. 
We briefly return to such naive definition in \S \ref{section: a version with multiplicities}; this reveals a pathology in terms of an explicit multiplicity $\ell^m$ -- see Theorem \ref{theorem: naive upshot}. 
\end{remark}

\begin{remark}
In the case $h = 1$, the space $\dot{X}_1$ is the perfection of the usual Deligne--Lusztig variety and $G_1 = \mathbf{GL}_n(\F_q)$. The complexes $\RG_c(\dot{X}_1, \Lambda)_{\eL \theta}$ were studied in \cite{BR02, BDR16}; they play an important role in the structure of $\D^b(G_1, \Fl)$.
\end{remark}

\subsection{Isotypic parts for whole Deligne--Lusztig spaces} \label{section: equivariant cohomology of deligne--lusztig spaces}
The goal of this section is to introduce the isotypic parts of cohomology of the whole Deligne--Lusztig varieties $\dot{X}$. More precisely, we provide a definition of the derived isotypic components
$$\RG_c(\dot{X}, \Lambda)_{\eL \theta} \in \D^b(G, \Lambda),$$
where $\Lambda$ is a coefficient ring and $\theta \in \Hom_{\Grp}(T, \Lambda^\times)$.

Since $\dot{X}$ is infinite-dimensional, the compactly supported cohomology does not a priori make sense. The actual construction proceeds as in \cite{Iva19} -- we take the cohomology at a finite level $\dot{X}_h$ for sufficiently big $h \geq 1$, and construct the cohomology of $\dot{X}$ out of it by an explicit representation-theoretic procedure.

Take $\dot{X}$, $\dot{X}_{\O}$, $\dot{X}_h$ resp. their quotients $X$, $X_{\O}$, $X_h$ as in \S \ref{section: the concrete case of GL_n} and $\Lambda$ as in Setup \ref{setting: coefficient rings}.
\begin{construction} \label{construction: cohomology of the whole space}
Let $\theta \in \Hom_{\Grp}(T, \Lambda^\times)$ and $h \geq \lvl(\theta)$. Then we define
$$\RG_c(\dot{X}, \Lambda)_{\eL \theta} := \cInd_{ZG_{\O}}^G \RG_c(\dot{X}_h, \Lambda)_{\eL \theta} \in \D^b(G, \Lambda)$$
where $\RG_c(\dot{X}_h, \Lambda)_{\eL \theta}$ is regarded as an element of $\D^b(ZG_{\O}, \Lambda)$ by inflation to $G_{\O}$ and by letting the center $Z$ of $G$ act via $\theta$. This is well-defined up to an even cohomological shift.
\end{construction}

\begin{discussion}\label{discussion: cohomology of the whole space}
Let us expand the construction in words and point out what needs to be checked. The restriction of the smooth character $\theta$ to $T_{\O}$ factors through any $T_h$ with $h \geq \lvl(\theta)$. By \S \ref{section: isotypic parts at finite level}, we get the derived isotypic part
$$\RG_c(\dot{X}_h, \Lambda)_{\eL \theta} \in \D^b(G_h , \Lambda).$$
We will show in Lemma \ref{lemma: representation stability of finite deligne--lusztig spaces} that the equivariant tower of the finite level Deligne--Lusztig spaces $\dot{X}_h$ is {\it representation stable}; the derived complex above is thus independent of the choice of $h \geq \lvl(\theta)$ up to an even cohomological shift.

Having a good notion of derived isotypic parts at finite level, we regard $\RG_c(\dot{X}_h, \Lambda)_{\eL \theta}$ as an object of $\D^b(G_{\O}, \Lambda)$ by inflating along the quotient map $G_{\O} \to G_h$. This makes good sense, because the inflation is exact.
Furthermore, we upgrade this to an object of $\D^b(ZG_{\O}, \Lambda)$ by letting the center $Z$ act via the desired character $\theta$ of $T$. 

Finally, we obtain the sought-for complex
$$\RG_c(\dot{X}_h, \Lambda)_{\eL \theta} \in \D^b(G, \Lambda)$$
by applying the functor $\cInd_{ZG_{\O}}^G$ of compact induction from the clopen subgroup $ZG_{\O}$. This makes good sense, because $\cInd_{ZG_{\O}}^G$ is exact by \S \ref{subsection: categories of smooth representations and their grothendieck groups}.
\end{discussion}

We have completely ignored the question of finiteness throughout the discussion. However, we will eventually need some sort of finiteness to have a good notion of Euler characteristic. We will discuss this thoroughly in \S \ref{section: finteness properties of the cohomology}. For the moment, we record one elementary lemma in this direction.
\begin{lemma}\label{lemma: finite type of cohomology of the whole space}
Let $\theta \in \Hom_{\Grp}(T, \Lambda^\times)$. Then $\RG_c(\dot{X}, \Lambda)_{\eL \theta}\in \D^b(G, \Lambda)$ is of finite type.
\end{lemma}
\begin{proof}
At the finite level, $\RG_c(\dot{X}_h, \Lambda)_{\eL \theta_h}$ is of finite type in $\D^b(G_h, \Lambda)$ by Lemma \ref{corollary: rickards complex for finite group actions}. This clearly remains true when we inflate to $G_{\O}$ and when we let $Z$ act via $\theta$, simply because we add more structure to the same complex of $\Lambda$-modules. Finally, compact induction from the clopen subgroup $ZG_{\O}$ preserves finite type by Lemma \ref{lemma: induction and finite type}, so that
$$\RG_c(\dot{X}, \Lambda)_{\eL \theta} = \cInd_{ZG_{\O}}^G \RG_c(\dot{X}_h, \Lambda)_{\eL \theta} \in \D^b(G, \Lambda)$$
is of finite type: it can be represented by a bounded complex with finite type terms.
\end{proof}

\subsection{Representation stability of the truncated Deligne--Lusztig spaces}
\label{section: representation stability of the truncated deligne--lusztig spaces}

To provide the missing input for Construction \ref{construction: cohomology of the whole space}, we discuss the representation stability of the sequence of spaces $\dot{X}_h$. This works up to an explicit even cohomological shift.

\begin{lemma}\label{lemma: representation stability of finite deligne--lusztig spaces}
Let $\theta \in \Hom_{\Grp}(T, \Lambda^{\times})$ and $h \geq h' \geq \lvl(\theta)$. Then we have a canonical identification
$$\RG_c(\dot{X}_{h}, \Lambda)_{\eL \theta} \cong \RG_c(\dot{X}_{h'}, \Lambda)_{\eL \theta}[2(n-1)(h'-h)] \in \D^b(G_h, \Lambda).$$ 
\end{lemma}

\begin{proof}
From Propositions \ref{proposition: T_h torsor} and \ref{proposition: affine fibration} we obtain a commutative diagram of schemes
\begin{equation*}
\begin{tikzcd}
\dot{X}_h \arrow{r}{g} \arrow[swap]{rd}{\pi} & \dot{X}_h/T^{h-1}_h \arrow{r}{f} \arrow{d}{\pi'} & \dot{X}_{h-1} \arrow{d}{\pi''} \\
& X_{h} \arrow{r}{\overline{f}} & X_{h-1}
\end{tikzcd}
\end{equation*}
The fibers of $f$ are perfections of affine spaces $\mathbb{A}^{n-1}$ over the base field $\overline{\F}_q$.
At the same time, the maps $g$ and $\pi$, resp. $\pi'$, $\pi''$ are finite quotient maps for the group actions of $T^{h-1}_h$ and $T_h$ resp. $T_{h-1}$. For any $\theta \in \Rep(T_{h'}, \Lambda)^{\ft} \subseteq \Rep(T_{h}, \Lambda)^{\ft}$, we can compute inside $\D^b(G_h, \Lambda)$ as follows.
\begin{equation*}
\begin{aligned}
\RG_c(\dot{X}_h, \Lambda) \otimes^{\eL}_{\Lambda[T_h]} \theta
& \cong  \RG_c(X_h, \pi_*\Lambda) \otimes^{\eL}_{\Lambda[T_h]} \theta &  \\
& \cong \RG_c(X_h, \pi_*\Lambda) \otimes^{\eL}_{\Lambda[T_{h}]} \Lambda[T_{h-1}] \otimes^{\eL}_{\Lambda[T_{h-1}]} \theta  & \\
& \cong \RG_c(X_h, \pi_*\Lambda \otimes^{\eL}_{\Lambda[T_{h}]} \Lambda[T_{h-1}]) \otimes^{\eL}_{\Lambda[T_{h-1}]} \theta & \\
& \cong  \RG_c(X_h, \pi'_*\Lambda) \otimes^{\eL}_{ \Lambda[T_{h-1}]} \theta \qquad & \\
& \cong  \RG_c(\dot{X}_h/T^{h-1}_h, \Lambda) \otimes^{\eL}_{ \Lambda[T_{h-1}]} \theta \qquad & \\
& \cong  (\RG_c(\dot{X}_{h-1}, \Lambda) \otimes^{\eL}_{ \Lambda[T_{h-1}]} \theta)[-2(n-1)] \qquad & \\
\end{aligned}
\end{equation*}

The first equality holds by \eqref{equation: pushforward to quotient}. The second is just the abusive identification of $\theta$ as representation of either $T_h$ or $T_{h-1}$; the second tensor product is actually nonderived. The third line holds by the projection formula for étale cohomology.

For the fourth line, we need to see that $\pi_*\Lambda \otimes^{\eL}_{\Lambda[T_{h}]} \Lambda[T_{h-1}] \cong \pi'_*\Lambda$. Here, we note that $\Lambda[T_{h-1}] \cong \Lambda[T_{h}/T^{h-1}_h]$ is projective as $\Lambda[T_h]$-module by Lemma \ref{lemma: projectivity of permutation modules} as $T^{h-1}_h$ is a $p$-group for $h \geq 2$. The desired isomorphism can be checked on stalks; by the above projectivity it reduces to the nonderived formula
$$\Lambda[T_h/C_{T_h}(x)] \otimes_{\Lambda[T_h]} \Lambda[T_{h-1}] \cong \Lambda[T_{h-1}/C_{T_{h-1}}(x)],$$
which is clear. (Here, $x$ is any geometric point of $X_h$ and $C_{T_h}(x)$ is its stabilizer with respect to our character $\theta$.)
The fifth line is then obtained in the same way as the first.

The sixth line holds because $f$ is a smooth map with fibers perfections of $\mathbb{A}^{n-1}$. Indeed, one can use smooth base change together with the canonical local identifications of the form $\RG_c(V~\times~\mathbb{A}^{n-1},~\Lambda) \cong \RG_c(V, \Lambda)\otimes^{\eL}_{\Lambda} \RG_c(\mathbb{A}^{n-1}, \Lambda) \cong \RG_c(V, \Lambda)[-2(n-1)]$. One should observe that this is compatible with the $T_h$-actions.
\end{proof}

\begin{remark}
In the fourth step, we employed the projectivity of $\Lambda[T_{h-1}]$. Alternatively, we can use that the $T_h$-action on $\dot{X}_h$ is free, showing the stalkwise projectivity of the other factor $\pi_*\Lambda$.
\end{remark}

\begin{remark}\label{remark: definition of RG(X)}
The above stabilization holds only up to an even degree shift -- as remarked in \cite[p. 10]{Iva19}, one can formally get rid of this shift. In the current context when all $\dot{X}_h$ are perfectly smooth, this is simply achieved by taking cohomology $\RG(\dot{X}_{h}, \Lambda)_{\eL \theta}$ without the compact support condition.  
\end{remark}

\subsection{Finiteness and Euler characteristic}
\label{section: finteness properties of the cohomology}

Construction \ref{construction: cohomology of the whole space} is given by a composition of several exact functors. It follows that whenever it preserves finite length of Euler characteristic, it specializes to the level of Grothendieck groups. 

\begin{construction}\label{construction: euler characteristic of the whole space}
\label{definition: euler characteristic of the whole space}
Let $h \geq \lvl(\theta)$ and assume that $\cInd_{ZG_{\O}}^G [\RG_c(\dot{X}_h, \Lambda)_{\eL \theta}]$ is a formal linear combination of finite length representations. Then it gives a well-defined element of $\G_0(G, \Lambda)$ which we denote 
\begin{equation*}
[\RG_c(\dot{X}, \Lambda)_{\eL \theta}] := \cInd_{ZG_{\O}}^G [\RG_c(\dot{X}_h, \Lambda)_{\eL \theta}] \in \G_0(G, \Lambda).
\end{equation*}
We alternatively denote this element by $[H^\bullet_c(\dot{X}, \Lambda)_{\eL \theta}]$. It is independent of the choice of $h \geq \lvl(\theta)$ by Lemma \ref{lemma: representation stability of finite deligne--lusztig spaces}.
\end{construction}

Beware that by our conventions, $\G_0(G, \Lambda)$ is the Grothendieck group of smooth {\it finite length} representations. A priori, we only know that $\RG_c(\dot{X}, \Lambda)_{\eL \theta}$ is of {\it finite type} by Lemma \ref{lemma: finite type of cohomology of the whole space}. This necessitates in the extra assumption that $\cInd_{ZG_{\O}}^G [\RG_c(\dot{X}_h, \Lambda)_{\eL \theta}]$ is a formal linear combination of finite length representations. 
Whenever $\RG_c(\dot{X}, \Lambda)_{\eL \theta} \in \D^b(\dot{X}, \Lambda)$ can be represented by a complex with finite length entries, its Euler characteristic gives rightaway an element of $\G_0(G, \Lambda)$ which agrees with Construction \ref{construction: euler characteristic of the whole space}. 

We now explain why Construction \ref{construction: euler characteristic of the whole space} applies to the cases considered in this paper -- see \S \ref{subsubsection: characteristic zero coefficients} for $\Lambda = \Ql$ and \S \ref{subsubsection: characteristic ell coefficients} for $\Lambda = \Fl$.

\subsubsection{Characteristic zero coefficients}\label{subsubsection: characteristic zero coefficients}

Employing results of \cite{Vig96} and \cite{Iva19}, we now show that Construction \ref{construction: euler characteristic of the whole space} applies for $[H^{\bullet}_c(\dot{X}, \Ql)_{\eL \theta}]$ when $\theta$ lies in general position. At the same time we keep track of integral structures.
This will imply the same finiteness for $[\RG_c(\dot{X}, \Fl)_{\eL \theta}]$ via reduction $r_\ell$ in \S \ref{section: comparison of isotypic parts}.

Since we work in characteristic zero, we can work non-derived. The main input is the following.

\begin{proposition}[{\cite[Theorem 5.1]{Iva19}}]\label{theorem: admissibility of construction}
Let $\theta \in \mathscr{X}(L, \overline{\Q}_\ell)$. Then $\pm [H^\bullet_c(\dot{X}, \overline{\Q}_\ell)_\theta]$ is admissible.
\end{proposition}

We now make the promised discussion; we also address the question of integrality. 
\begin{proposition}\label{proposition: finite length of H(X, Ql)}
For any $\theta \in \mathscr{X}(L, \overline{\Q}_\ell)^{\integral}$ we have a well-defined element
$$[H^\bullet_c(\dot{X}, \overline{\Q}_\ell)_\theta] \in \G_0(G, \overline{\Q}_\ell)^{\integral}.$$ 
\end{proposition}
\begin{proof}
Denote $h = \lvl(\theta)$, so that $[H^\bullet_c(\dot{X}, \overline{\Q}_\ell)_\theta] = \cInd_{ZG_{\O}}^G [H^\bullet_c(\dot{X}_h, \overline{\Q}_\ell)_\theta]$ holds by definition. We want to prove that this lands in finite length representations, and that it is integral.

Let us first discuss why $[H^\bullet_c(\dot{X}, \overline{\Q}_\ell)_\theta]$ is of finite length. By Lemma \ref{lemma: finite type of cohomology of the whole space} we already know that $[H^\bullet_c(\dot{X}, \overline{\Q}_\ell)_\theta]$ is of finite type. 
We now use results of \cite{Iva19}, which are valid under the assumption that $\theta$ lies in general position. By \cite[Corollary 3.3]{Iva19}, the finite level class $[H^\bullet_c(\dot{X}_h, \Ql)_\theta]$ is up to sign an irreducible representation, which we call $V$ for the moment. By Proposition \ref{theorem: admissibility of construction}, $\cInd_{ZG_O}^G V$ is a finite direct sum of irreducible supercuspidal representations, i.e. it is admissible.
Now by \cite[II.5.10]{Vig96} admissibility and finite type together imply finite length, as desired.

Secondly, we address the integrality of $H^\bullet_c(\dot{X}, \overline{\Q}_\ell)_\theta$. Again, $H^\bullet_c(\dot{X}_h, \overline{\Q}_\ell)_\theta$ is obviously an integral $G_h$-representation. 
The inflation to $G_{\O}$ clearly preserves integrality. Also prolonging to $ZG_{\O}$ by $\theta$ doesn't hurt the choice of integral structure by the integrality assumption on $\theta$. Finally, $\cInd^{G}_{ZG_{\O}}$ sends an integral structure of the original representation to an integral structure of the induced representation by the clopeness of $ZG_{\O}$ and Lemma \ref{lemma:induction and integrality}.
\end{proof}

\subsubsection{Characteristic \texorpdfstring{$\ell$}{l} coefficients}\label{section: comparison of isotypic parts}
\label{subsubsection: characteristic ell coefficients}

We now show that Construction \ref{construction: euler characteristic of the whole space} applies for $[H^{\bullet}_c(\dot{X}, \Fl)_{\eL \psi}]$ with $\psi$ in general position. At the same time, we check that this is given by $r_\ell [H^\bullet_c(\dot{X}, \overline{\Q}_\ell)_{\theta}]$ for any lift $\theta$ of $\psi$.

\begin{lemma}\label{Tsum}
\label{lemma: comparison of isotypic parts}
Let $\psi \in \mathscr{X}(L, \Fl)$ and $\theta \in \mathscr{X}(L, \overline{\Q}_\ell)^{\integral} $ one of its lifts under $r_\ell$. Then
\begin{equation*}
 r_\ell [H^\bullet_c(\dot{X}, \overline{\Q}_\ell)_{\theta}] = [H^\bullet_c(\dot{X}, \overline{\F}_\ell)_{\eL \psi}] \in \G_0(G, \overline{\F}_\ell).   
\end{equation*}
\end{lemma}
\begin{proof}
By Lemma \ref{Tlifts} and subsequent discussion, $\psi|_{T_{\O}}$ has precisely $\ell^m$ lifts under $r_\ell$, each of which can be non-uniquely extended to a desired lift $\theta$ from the statement. Such a lift has the same level $h = \lvl(\psi) = \lvl(\theta)$.
At this finite level, Corollary \ref{corollary: equivariant euler characterictic and reduction} yields an equality
$$r_\ell [H^\bullet_c(\dot{X}_h, \overline{\Q}_\ell)_{\theta}] = [H^\bullet_c(\dot{X}_h, \overline{\F}_\ell)_{\eL \psi}].$$

Now we go through Construction \ref{construction: euler characteristic of the whole space} of the Euler characteristic of $\dot{X}$. First we inflate to $G_{\O}$; this operation is additive and commutes with $r_\ell$. Then we extend to $ZG_{\O}$-representations by letting $Z$ act via $\psi$; this is again additive and can be commuted through $r_\ell$ to the action of $\theta$ on the left-hand side. 

Finally, we apply $\cInd_{ZG_{\O}}^G$ to both sides; this is additive and commutes with $r_{\ell}$. More precisely, Lemma \ref{lemma: cind and rl} simultaneously shows that Construction \ref{construction: euler characteristic of the whole space} applies to $[H^\bullet_c(\dot{X}, \overline{\F}_\ell)_{\eL \psi}]$ and that
\begin{equation*}
r_\ell [H^\bullet_c(\dot{X}, \overline{\Q}_\ell)_{\theta}] = [H^\bullet_c(\dot{X}, \overline{\F}_\ell)_{\eL \psi}].    
\end{equation*}
Indeed, during the entire computation, we stayed in the Grothendieck groups of finite length integral representations. Up to the final step, this is clear by the choice of $\theta$ and finite dimensionality of the representations. Finally, $[H^\bullet_c(\dot{X}, \overline{\Q}_\ell)_{\theta}]$ is integral of finite length by Proposition \ref{proposition: finite length of H(X, Ql)}. In particular, the application of Lemma \ref{lemma: cind and rl} is correct for the $ZG_{\O}$-representation constructed above;
the reduction maps $r_\ell$ indeed land in the correct targets by Discussion \ref{reductionmodulolitems} (iii) resp. (ii). 
\end{proof}

\subsection{Realization of the characteristic zero local Langlands correspondence}\label{section: realization of the local langlands correspondence}
We briefly summarize the main results of \cite{Iva18, Iva19, Iva20a, Iva20b} regarding the relationship of the isotypic parts $[H^\bullet_c(\dot{X}, \Ql)_\theta]$ to the local Langlands correspondence with $\Ql$-coefficients.

Let $\theta \in \Irr(T, \overline{\Q}_\ell)$. The Howe decomposition asserts the existence of a unique tower of subfields $K = L_0 \subsetneq L_1 \subsetneq \dots \subsetneq L_{t-1} \subsetneq L_t = L$ such that 
$$\theta = (\theta_0 \circ N_{L/L_0})\cdot (\theta_1 \circ N_{L/L_1}) \cdots (\theta_0 \circ N_{L/L_{t-1}})\cdot (\theta_t)$$
for some characters $\theta_i: L_i^\times \to \overline{\Q}_\ell$ of unique levels.
From this, we get numerical invariants of $\theta$. Regarding $\theta$ as an element of $\Irr(T_h, \overline{\Q}_\ell)$ for a fixed $h \geq \lvl (\theta)$, \cite[Theorem 6.1.1]{Iva20a} defines an integer $\cd(\theta)$ (called $r_{\chi}$ there) by an explicit formula dependent only on the group $G$ and $\theta|_{T^1}$.
\begin{proposition}\label{Tdegree}
The function $\cd: \Irr(T_h, \overline{\Q}_\ell) \to \Z$ factors through the reduction $r_\ell: \Irr(T_h, \overline{\Q}_\ell) \to  \Irr(T_h, \overline{\F}_\ell)$. 
We use the same notation $\cd: \Irr(T_h, \overline{\F}_\ell) \to \Z$ for this induced map.
\end{proposition}
\begin{proof}
By the discussion in \S \ref{Tsec -- prof}, the lifts of given $\psi \in \Irr(T, \overline{\F}_\ell)$ along $r_\ell$ differ at most by a character of $T_1 = \O^\times_L/T^1$, i.e. the restrictions of all these lifts to $T^1$ are equal. The statement now follows from the definition of $\cd$ discussed above.
\end{proof}

\begin{theorem}[\cite{Iva19}]\label{Iva}
\label{theorem: ivanov realization of local langlands}
Assume $p > n$. Let $\theta \in \mathscr{X}(L, \overline{\Q}_\ell)^{\sg}$. Then $M(\theta) := (-1)^{\cd(\theta)} \cdot [H^{\bullet}_c(\dot{X}, \overline{\Q}_\ell)_\theta]$ is an irreducible supercuspidal representation of $G$ and the association
$\sigma(\theta) \leftrightarrow M(\theta)$
is a partial realization of the local Langlands correspondence.
\end{theorem}
\begin{proof}
Up to the explicit determination of the sign $\pm[H^{\bullet}_c(\dot{X}, \overline{\Q}_\ell)_\theta]$, this is \cite[Theorem A]{Iva19}.  

The sign is discussed in the course of \cite[\S 7.1]{Iva19}. By \cite[Theorem 6.1.1]{Iva20a}, we get the definition of $\cd$, which returns the single nonzero degree of the $\theta|_{T^1}$-isotypic component of the cohomology of certain closed subscheme $X^1_h$ of our $\dot{X}_h$. Therefore it gives also the single nonzero degree of the $\theta$-isotypic component of another closed subscheme $X_{h, n'}$, which is given by a finite disjoint union of copies of $X^1_h$, see \cite[\S 2]{Iva20a}. 
By \cite[Corollary 4.2]{Iva19}, there is an identification $[H^\bullet_c(X_{h, n'}, \overline{\Q}_\ell)_\theta] = [H^\bullet_c(\dot{X}_h, \overline{\Q}_\ell)_\theta]$ inside $\G_0(G_h, \overline{\Q}_\ell)$ and consequently $(-1)^{\cd (\theta)}\cdot [H^\bullet_c(\dot{X}_h, \overline{\Q}_\ell)_\theta]$ is indeed an (irreducible) representation; the same sign then appears for cohomology of the whole $\dot{X}$ by design.
\end{proof}

\begin{remark}
It is expected that this holds for any $\theta \in \mathscr{X}(L, \overline{\Q}_\ell)$, i.e. after replacing ``strongly general position" by ``general position".
\end{remark}

\subsection{Main result}
\label{section: main result}
We are now ready to piece everything together and deduce our main result.

\begin{theorem}\label{theorem: upshot}
Assume $p > n$. Let $\psi \in \mathscr{X}(L, \Fl)^{\sg}$. Then:
\begin{itemize}
\item[(1)] the Weil representation $\sigma(\psi)$ is irreducible,
\item[(2)] the representation $M(\psi):=(-1)^{\cd(\psi)} \cdot [H^\bullet_c(X, \Fl)_{\eL \psi}]$ is irreducible supercuspidal, 
\item[(3)] the association $\sigma(\psi) \longleftrightarrow M(\psi)$ partially realizes the modular local Langlands correspondence.
\end{itemize} 
\end{theorem}
\begin{proof}
Fix an integral lift $\theta \in \Irr(T, \Ql)^{\sg , \integral}$ of $\psi \in \Irr(T, \Fl)^{\sg}$ as in Lemma \ref{lemma: lifting characters}; such $\theta$ lies in strongly general position by Lemma \ref{lemma: strongly general position and reduction}.
As $\psi$ in particular lies in general position, Proposition \ref{proposition: weil induction} shows that $\sigma(\psi)$ is irreducible. By Lemma \ref{lemma: weil induction and reduction} we have $r_\ell(\sigma(\theta)) = \sigma(\psi)$.

Using the assumption $p>n$, Theorem \ref{Iva} applies to the strongly general character $\theta$. Thus $M(\theta) := \pm [H^\bullet_c(X, \Ql)_{\theta}]$ is an irreducible supercuspidal representation of $G$ and the sign identifies as $(-1)^{\cd(\theta)}$. By Lemma \ref{lemma: comparison of isotypic parts}, reduction modulo $\ell$ returns the modular representation in question:
\begin{equation}\label{equation: reduction}
r_\ell \left( M(\theta) \right) = (-1)^{\cd(\theta)} \cdot r_\ell [H^\bullet_c(X, \Ql)_{\theta}] = (-1)^{\cd(\psi)} \cdot [H^\bullet_c(X, \Fl)_{\eL \psi}] = M(\psi).
\end{equation}
Moreover, still by Theorem \ref{theorem: ivanov realization of local langlands}, the association
$$\sigma(\theta) \mapsfrom \theta \mapsto M(\theta) = (-1)^{\cd(\theta)} \cdot [H^\bullet_c(\dot{X}, \Ql)_{\theta}]$$
realizes the local Langlands correspondence over $\overline{\Q}_\ell$.

Now Theorem \ref{theorem: vigneras} applies: since $\sigma(\theta)$ is sent by the local Langlands correspondence to $M(\theta)$, its irreducible reduction $\sigma(\psi)$ is sent by the modular local Langlands correspondence to $r_\ell (M(\theta))$, which is thus irreducible supercuspidal. 
Since $r_\ell (M(\theta)) = M(\psi)$ by \eqref{equation: reduction}, the proof is finished.
\end{proof}

\subsection{A version with multiplicities}
\label{section: a version with multiplicities}

In this complementary section, we present a variation on Theorem \ref{theorem: upshot} coming from a naive definition of isotypic parts. The proof of this version goes along the same lines as the proof of Theorem \ref{theorem: upshot}.

When $\ell \nmid |T_1|$, the category $\rep(T_h, \Lambda)$ is semisimple and these naive isotypic parts coincide with the derived ones. When $\ell \mid |T_1|$, an extra multiplicity $\ell^m$ appears in our theorem. This discrepancy reveals that the complexes $\GG_c(\dot{X}_h, \Fl)$ contain interesting extensions. The reader may compare this with the results of \cite{BR02, BDR16}. 

\begin{notation}
Let $\Lambda$ be an algebraically closed field and $G$, $G'$ two finite groups. Then there is a bijection \cite[Theorem 10.33]{CuRe}
$$\Irr(G, \Lambda) \times \Irr(G', \Lambda) \xrightarrow{\cong}  \Irr(G\times G', \Lambda), \qquad (V, V') \mapsto V \otimes_{\Lambda} V' $$ 
inducing canonical isomorphism of free abelian groups $\G_0(G\times G', \Lambda) \cong \G_0(G, \Lambda) \otimes_{\Z} \G_0(G', \Lambda)$.

We equip each $\G_0$ with the scalar product making the classes of irreducible representations into an orthonormal basis. For any $V' \in \Irr(G', \Lambda)$ we then have the orthonormal projection
$$(-)[V']: \G_0(G\times G', \Lambda) \to \G_0(G, \Lambda), \qquad [M] \mapsto [M][V']$$
We call it the {\it naive $V'$-isotypic part}. 
\end{notation}

Considering this for the groups $G_h \times T_h$ of our interest, we now address the compatibility of naive isotypic parts with the reduction map $r_\ell$.
\begin{lemma}\label{lemma: naive isotypic parts and reduction}
Let $M \in \G_0(G_h \times T_h, \overline{\Q}_\ell)$. Take $\psi \in \Irr(T_h, \overline{\F}_\ell)$ and let $\theta_i \in \Irr(T_h, \overline{\Q}_\ell)$ for $i \in I = \{1, \dots, \ell^m \}$ be its finitely many lifts. Then
$$\sum_{i \in I}r_\ell\left(M[\theta_i]\right) = r_\ell(M)[\psi].$$
\end{lemma}
\begin{proof}
The map $r_\ell$ is additive, as well as the formation of an isotypic part. Hence both sides of the equality are additive in $M$. So we only need to check the equality on the generators $[V \otimes_{\Lambda} V']$ from above, where it obviously holds.
\end{proof}

\begin{remark}
Altogether, one can define the naive isotypic parts of the cohomology of $\dot{X}$ as in Construction \ref{construction: euler characteristic of the whole space}: for $\theta \in \Irr(T, \Lambda)$ and $h \geq \lvl(\theta)$, this naive isotypic part is given by
$$[H^{\bullet}_c(\dot{X}, \Lambda)][\theta] :=  \cInd_{ZG_{\O}}^G [H^\bullet_c(\dot{X}_h, \Lambda)][\theta] \in \G_0(G, \Lambda).$$

The well-definedness of this variant can be bootstrapped from what we've already seen. For instance, consider the independence on $h$. We know this holds for $\Lambda = \Ql$, because the naive and derived isotypic parts coincide. The case of $\Lambda = \Fl$ then follows by reduction modulo $\ell$ through Corollary \ref{remark: euler characteristic and reduction}.
Similar reasoning together with Lemma \ref{lemma: naive isotypic parts and reduction} shows that $[H^{\bullet}_c(\dot{X}, \overline{\F}_\ell)][\theta]$ lands in the correct Grothendieck group $\G_0(G, \Fl)$. 
\end{remark}

Replacing derived isotypic parts by the naive ones, our main theorem transforms as follows.
\begin{theorem}\label{theorem: naive upshot}
In the setting of Theorem \ref{theorem: upshot}, the irreducible supercuspidal representation $M(\psi)$ is given by
$$[M(\psi)] = (-1)^{\cd(\psi)} \cdot \frac{1}{\ell^m} \cdot [H^{\bullet}_c(\dot{X}, \overline{\F}_\ell)][\psi].$$
\end{theorem}
In other words, $[H^{\bullet}_c(\dot{X}, \overline{\F}_\ell)][\psi] \in \G_0(G, \Fl)$ is an $\ell^m$ multiple of the class of $M(\psi)$. When $\ell$ doesn't divide $|T_1| = q^n - 1$, the multiplicity is $1$; this is the case when $\rep(T_h, \Fl)$ is semisimple. When $\ell$ divides $|T_1|$, this multiplicity becomes nontrivial. 
This discrepancy shows that the complexes $\RG(\dot{X}_h, \Fl)$ and $\GG(\dot{X}_h, \Fl)$ contain interesting information for small $\ell$.

\begin{proof}
Denote by $\theta_i$ for $i = 1, \dots, \ell^m$ lifts of $\psi$ under $r_\ell: ~ \Irr(T, \overline{\Q}_\ell)^{\integral} \to \Irr(T, \overline{\F}_\ell)$, whose restrictions to $T_{\O}$ are pairwise distinct as in Remark \ref{remark: integral lifts of characters}.

The proof of Theorem \ref{theorem: upshot} applies for each of these $\theta_i$ separately, giving irreducible supercuspidal representations $M(\theta_i) := (-1)^{\cd (\psi)} \cdot [H^\bullet_c(\dot{X}, \overline{\Q}_\ell)][{\theta_i}]$ of $G$ -- indeed, the naive and derived isotypic parts coincide here. By the arguments in that proof, each of these representations reduces to $ M(\psi)$ modulo $\ell$, which is independent of $i = 1, \dots, \ell^m$.
By Lemma \ref{lemma: naive isotypic parts and reduction} and additivity of $r_\ell$ we infer
\begin{equation*}
(-1)^{\cd(\psi)} \cdot [H^\bullet_c(\dot{X}, \overline{\F}_\ell)][\psi] = \sum_{i=1}^{\ell^m} r_\ell ((-1)^{\cd(\psi)}[H^\bullet_c(\dot{X}, \overline{\Q}_\ell)][\theta_i]) = \sum_{i=1}^{\ell^m} r_\ell [M(\theta_i)] = \ell^m \cdot [M(\psi)]. 
\end{equation*}
\end{proof}

\begin{remark}
This naive version can be deduced without Rickard's complexes. 
Indeed, only Corollary \ref{remark: euler characteristic and reduction} is needed and this can be proved directly.
\end{remark}
\printbibliography

\bigskip

\noindent Jakub Löwit, \newline
Institute of Science and Technology Austria (ISTA), \newline
Am Campus 1, \newline 
3400 Klosterneuburg, \newline
Austria \newline
\texttt{jakub.loewit@ist.ac.at}

\end{document}